\documentclass[11pt,leqno]{amsart}

\usepackage{amsmath}
\usepackage{amsthm}
\usepackage{graphicx}
\usepackage{mathtools}
\usepackage{amsfonts}
\usepackage{dsfont}
\usepackage{amssymb}
\usepackage{hyperref}
\usepackage{bm}
\usepackage[margin=1 in]{geometry}
\usepackage{enumerate}
\usepackage[normalem]{ulem}
\usepackage{tikz}
\usepackage{tikz-cd}
\usepackage{xcolor}
\usetikzlibrary{matrix,arrows,positioning,automata}
  \usepackage{cancel}
  \usepackage{todonotes}
\usepackage{dsfont}
\usepackage{comment}
  
\numberwithin{equation}{section}

\newcommand{\R}{\mathbb R}
\newcommand{\T}{\mathbb{T}}
\def\E{\mathbb E}

\newcommand{\bx}{\mathbf{x}}
\newcommand{\bX}{\mathbf{X}}
\newcommand{\balpha}{\mathbf{\alpha}}
\newcommand{\bq}{\mathbf{q}}
\newcommand{\bp}{\mathbf{p}}
\newcommand{\bB}{\mathbf{B}}

\numberwithin{equation}{section}
\newtheorem{thm}{Theorem}[section]
\newtheorem{lem}[thm]{Lemma}

\newtheorem{prop}[thm]{Proposition}

\theoremstyle{definition}

\newtheorem{rmk}[thm]{Remark}

\def\ds{\displaystyle}

\def\div{\operatorname{div}}
\def\tr{\operatorname{tr}}
\def\bx{{\bf x}}

\def\ba{{\bf a}}
\def\bA{{\bf A}}
\def\bX{{\bf X}}
\def\balpha{{ \bf \alpha}}
\def \lg{\langle}
\def \rg{\rangle}

\newcommand{\norm}[1]{\left\Vert#1\right\Vert}
\renewcommand{\bar}{\overline}
\renewcommand{\tilde}{\widetilde}

\title[Particle Approximation for Conditional Control with Soft Killing]{Particle Approximation for Conditional Control with Soft Killing}

\date{\today}

\begin{document}

\author[R. Carmona]{René Carmona
\address{(R. Carmona) Department of Operations Research \& Financial Engineering, Princeton University, Princeton NJ 08544, USA}\email{rcarmona@princeton.edu}}

 \author[S. Daudin]{Samuel Daudin 
\address{(S. Daudin) Universit\'e Paris Cité, Laboratoire Jacques-Louis-Lions, Paris, France
}\email{samuel.daudin@u-paris.fr
}}

\begin{abstract}
The aim of this paper is to develop a particle approximation for the conditional control problem introduced by P.-L. Lions during his lectures at the Collège de France in November 2016. We focus on a \textit{soft killing} relaxed version of the problem, which admits a natural counterpart in terms of stochastic optimal control for a large number of interacting particles. Each particle contributes to the overall population cost through a time-evolving weight that depends on the trajectories of all the other particles. Using recently developed techniques for the analysis of the value function associated with the limiting problem, we establish sharp convergence rates as the number of particles tends to infinity.
\end{abstract}

\setcounter{tocdepth}{1}

\thanks{R. Carmona was partially supported by AFOSR under Grant No.FA9550-23-1-0324.}

\thanks{ S. Daudin  
acknowledges the financial support of the European Research Council (ERC) under the European Union’s Horizon Europe research and innovation program (AdG
ELISA project, Grant Agreement No. 101054746). Views and opinions expressed are however
those of the author(s) only and do not necessarily reflect those of the European Union or the
European Research Council Executive Agency. Neither the European Union nor the granting
authority can be held responsible for them. }

\maketitle

\section{Introduction}

Practical applications requiring the optimization of cost constrained systems abound, and 
a natural modelling approach to these problems is to work within the framework of the optimal control of stochastic processes with conditional costs. To the best of our knowledge, the first mathematical formulation of such problems was introduced by P.L. Lions in his lectures at the Coll\`ege de France in November 2016. See \cite{Lions2016}. As stated, the problem did not fit into the standard framework of stochastic control problems studied in the literature, so its solution required new ideas and even new technology. In his lectures, Lions emphasized the special role of the conditioning in the evaluation of the objective function to be optimized, and conjectured that this singular feature could invalidate the common \emph{folk theorem} in optimal control according to which optimization over Markovian and general open loop controls would give the same value.

While the closed-loop version of the large time limit and the \emph{finite volume limit} of the problem were discussed in \cite{AchdouLauriereLions}, the first rigorous comparison of the optimizations over closed-loop and open-loop controls was given in \cite{cc_IJM}  There, the authors considered the \emph{soft killing}  version of the problem and provided complete answers to the comparison of the optimizations over different families of controls. They reformulated the problem as an equivalent (thanks to a tailor-made superposition principle) deterministic control problem over a (nonlinear and nonlocal) dynamical system on the Wasserstein space of probability measures, and they characterizated the optimal solution by a forward backward system of Partial Differential Equations (PDEs) in the spirit of the Mean Field Game (MFG) systems.

In the present paper, we introduce a particle approximation for the problem, following recent advances in the theory of Mean-Field Optimal Control models. This theory, together with the related theory of Mean-Field Games, addresses optimal control problems for large populations of interacting particles, see \cite{CarmonaDelarue_book_I,CarmonaDelarue_book_II,CardaliaguetDelarueLasryLions} for introductory textbooks on the subject. Understanding the infinite-particle limit and obtaining precise quantitative estimates are challenging tasks that have recently attracted significant attention, see e.g. \cite{budhiraja2012,Lacker2017,DjetePossamaiTan,germain_pham_warin_2022,ddj2023}. In our setting, however, we have to face an additional difficulty: each particle carries an individual weight that evolves over time, depending on the trajectories of all the other particles. Despite this added difficulty, we extend the approach developed in \cite{ddj2023} and establish convergence rates. The starting point is the observation that the value function of the limiting problem—formulated over a space of probability measures—when projected onto a finite-dimensional space, formally satisfies a PDE similar to the one governing the value function for a finite number of particles. A comparison argument then yields an estimate of the difference between the two value functions. However, the value function in the limit lacks the necessary smoothness for a direct application of this method. To overcome this, we introduce a suitable regularization of the limiting value function and apply the projection argument to this smoothed version.

Since the completion of \cite{cc_IJM}, a significant interest appeared in the mean field models of McKean-Vlasov type where the interaction occurs through conditional distributions as opposed to plain distributions. First, the equality of the optimization values over closed-loop and open-loop controls was proved using a new conditional mimicking argument in \cite{cl_mimicking}.
Next, we can cite the recent works
\cite{HamblyJettkant_cc} and \cite{Zhang},
as well as \cite{jettkant} which, like the present paper, emphasizes special connections of those models with particle systems.

The paper is organized as follows. In Section \ref{sec:statementofthepb} we introduce the optimal control problems and state our main results. In Section \ref{sec:AnalysisVF} we provide a first analysis of the value function of the limiting problem. In particular we prove that it is Lipschitz and semi-concave with respect to appropriate weak metrics. This result relies on the optimality conditions for the control problem that are postponed to Section \ref{sec:FOC}. Section \ref{sec:formalargumentprojection} contains the formal projection argument leading to the rates of convergence under the assumption that the limiting value function is \textit{smooth enough}. We present the mollification procedure in Section \ref{sec:supconvolandproperties16/04} and Sections \ref{sec:hardineq}, \ref{sec:easyineq} are then devoted to the main estimates of the paper.

\section{Statement of the problem and main result}

\label{sec:statementofthepb}

Without much loss of generality, we shall assume that the running cost is merely a quadratic kinetic energy term, and the terminal cost is given by a function $g : \T^d \rightarrow \R$. Purely for convenience, we assume that the state space of our dynamical system is the torus $\T^d$. Again, there is no loss of generality compared to the original problem in \cite{Lions2016} or \cite{cc_IJM}. Indeed, one can assume that the torus is large enough to strictly contain the bounded open domain in which the problem was set-up in \cite{Lions2016} or \cite{cc_IJM}.
The \emph{soft killing} feature of the model is provided by a \emph{potential function} $V : \T^d \rightarrow \R^+$ that gives the rates at which the particles are killed as a function of their positions.
We study the following $N$-particle system. Given a filtered probability space $(\Omega, \mathcal{F}, \mathbb{F}, \mathbb{P})$ supporting $N$ $d$-dimensional independent $\mathbb{F}$-Brownian motions $\bB_t := (B_t^1, \dots, B_t^N)$, we seek to minimize, over a finite time horizon $[0,T]$,
\begin{equation} 
J^N(t_0, \bx_0, \ba_0, (\balpha_t)_{t \geq t_0}) := \ds \E \Bigl[ 
\int_{t_0}^T \sum_{i=1}^N  \frac{1}{2 N_t^{i}}  |\alpha_t^{i}|^2 dt + \sum_{i=1}^N \frac{1}{N_T^{i}}g(X_t^{i})   \Bigr] 
\label{eq:defncostJN228/04}
\end{equation}
over the set of progressively measurable controls $(\alpha_t^1)_{t \geq t_0},\dots,(\alpha_t^N)_{t \geq t_0}$ in $L^2(\Omega \times [t_0,T]; \R^d)$ with $\bX_t := (X_t^1, \dots, X_t^N)$ and $ \bA_t := (A_t^1, \dots, A_t^N)$ solutions of
\begin{equation} 
\label{fo:N-state_dynamics}
\left \{
\begin{array}{ll}
dX_t^{i} = \alpha_t^{i} dt + dB_t^{i}, \quad X_{t_0}^{i} = x_0^{i} \in \T^d \\
dA_t^{i} = V(X_t^{i})dt, \quad A_{t_0}^{i} = a_0^{i} \in \R \\
\end{array}
\right.
\end{equation}
where the weights $N^i_t$ at time $t$ are defined by
\begin{equation}
    \label{fo:weightsSto}
    \frac{1}{N_t^{i}} = \frac{\exp(-A_t^{i})}{\sum_{j=1}^N \exp(-A_t^j)}. 
\end{equation}
Notice that the variables $X^1_t, \dots,X_t^N$ live in $\T^d$ for all $t \in [t_0,T]$ while the variables  $A_t^1 ,\dots, A_t^N$ live in $\R$ for $t \in [t_0,T]$.

In order to put things in perspective, we highlight two special cases for the potential function $V$. When $V$ is a constant independent of $x$ and $\ba_0 = \bigl( 0, \dots,0 \bigr)$ then $1/N_t^{i} = 1/N$ for all $t \geq t_0$ and all $1 \leq i \leq N$, and we are in the classical mean-field setting. 
Alternatively, if $V$ is given by: 
\begin{equation}
   V(x) = 
    \left \{
    \begin{array}{ll}
        0 & \mbox{ if } x \in D \\
        + \infty & \mbox{ if } x \notin D
    \end{array}
    \right.
\label{eq:V0ouinfty28/04}
\end{equation}
for some open subset $D$ of $\T^d$, then $\frac{1}{N_t^{i}} = \frac{1}{ N_t} \mathbf{1}_{\tau^{i}_{D} >t} $ where $\tau^{i}_{D}$ is the first exit time of the $i$-th particle from the domain $D$, and $N_t := \sum_{j=1}^N  \mathbf{1}_{ \tau^j_{D} >t} $ is the number of particles that did not attempt to exit the domain before time $t$. This corresponds to the \textit{hard killing} regime addressed in \cite{Lions2016, AchdouLauriereLions}. Here, the potential function $V$ will be taken as a regular function (thus excluding the \textit{hard killing} regime).

For $N \geq 1$, we define the value function
$$
v^N(t_0, \bx_0, \ba_0) := \inf_{\balpha} J^N(t_0, \bx_0, \ba_0, (\alpha_t)_{t \geq t_0}).$$
For convenience we introduce the notation
\begin{equation}
    \label{fo:weights}
n^{i}[\ba] := e^{a^{i}} \sum_{j=1}^N e^{-a^{j}},  \quad 1 \leq i \leq N, \quad \ba \in \R^N, 
\end{equation}
so that $N^i_t=n^i[\bA_t]$.
The Hamiltonian of the system is given, for $\bx=(x^1, \dots,x^N) \in (\T^d)^N$, $\ba \in (\R)^N$,  $\bp =(p^1, \dots p^N) \in (\R^d)^N$ and $\bq = (q^1, \dots, q^N) \in \R^N$ by
\begin{align*} 
H^N(\bx,\ba, \bp,\bq)  &=\sup_{(\alpha^1, \dots, \alpha^N) \in (\R^d)^N} \Bigl \{ - \sum_{i=1}^N \alpha^{i} \cdot p^{i} - \sum_{i=1}^N V(x^{i})q^{i} - \frac{1}{2} \sum_{i=1}^N \frac{1}{n^{i}[\ba]} |\alpha^{i}|^2 \Bigr \} \\
&= - \sum_{i=1}^N V(x^{i})q^{i} +  \sum_{i=1}^N  \frac{n^{i}[\ba]}{2} |p^{i}|^2,
\end{align*}
and the supremum is achieved by taking $\alpha^{i}= - n^{i}[\ba] p^{i}$ for $1 \leq i \leq N$.
We are dealing with a standard stochastic optimal control problem in finite dimension, see \cite{flemingsoner}, and therefore, by dynamic programming, $v^N$ is the unique  viscosity solution to the semilinear HJB equation
\begin{equation}
\left \{
\begin{array}{ll}
 \displaystyle    -\partial_t v^N - \frac{1}{2} \sum_{i=1}^N \Delta_{x^{i}}v^N - \sum_{i=1}^N V(x^{i}) \partial_{a^{i}}v^N +  \sum_{i=1}^N  \frac{n^{i}[\ba]}{2} |\nabla_{x^{i}} v^N|^2 = 0  \mbox{         in } [t_0,T]\times (\T^d)^N \times (\R)^N, \\ 
\displaystyle    v^N(T, \bx, \ba) =  \sum_{i=1}^N \frac{1}{n^{i}[\ba]} g(x^{i}) \quad \mbox{ in } (\T^d)^N \times (\R)^N. 
\end{array}
\right.
\end{equation}

In order to analyze the infinite particle limit of the system, we consider the following optimal control problem for a non-local non-linear Fokker-Planck equation
\begin{equation} 
\mathcal{V}(t_0, \mu_0) := \inf_{(\mu,\alpha)} \int_{t_0}^T \int_{\T^d} \frac{1}{2} |\alpha_t(x)|^2 d\mu_t(x)dt + \int_{\T^d} g(x)d \mu_T(x),
\label{def:limitvfV}
\tag{Pb}
\end{equation}
where the infimum is taken over the couples $(\mu, \alpha)$ with $\mu \in \mathcal{C}([t_0,T], \mathcal{P}(\T^d))$ and $\alpha \in L^{\infty}( \mu_t \otimes dt; \R^d)$, (ie $\alpha : [t_0,T] \times \T^d \rightarrow \R^d$ is measurable and there exists $C>0$ such that $ \mu_t \otimes dt  \{ (t,x) \in [t_0,T] \times \T^d , \quad |\alpha_t(x) | \geq C  \} = 0 $ ) satisfying
\begin{equation} 
\partial_t \mu_t - \frac{1}{2} \Delta\mu_t +\div(\alpha_t \mu_t) + \bigl( V(x) - \langle V;\mu_t \rangle \bigr)\mu_t = 0, \mbox{ in } (t_0,T) \times \T^d, \quad \mu_{t_0} = \mu_0,
\label{eq:FPEIntro-28/04}
\end{equation}
in the sense of distributions. This control problem is precisely the one investigated in \cite{CLL23} as a relaxed version of the original conditional control problem of \cite{Lions2016,AchdouLauriereLions}. The solution $(\mu_t)_{t \in [t_0,T]}$ of \eqref{eq:FPEIntro-28/04} is given by the probabilistic representation
$$ \mu_t = \E \bigl[ \frac{e^{-A_t} \delta_{X_t}}{\E [e^{-A_t}]}  \bigr] $$
where $(A_t,X_t)$ solves the Stochastic Differential Equation (SDE)
$$ 
dX_t = \alpha_t(X_t) dt + dB_t, \quad  dA_t = V(X_t) \quad t \geq t_0 
$$
where $(B_t)_{t \geq t_0}$ is a $d$-dimensional $\mathbb{F}$-Brownian motion. The cost in \eqref{def:limitvfV} can then be rewritten 
\begin{equation} 
\E \Bigl[ \int_{t_0}^T \frac{1}{2} |\alpha_t(X_t)|^2 q_tdt + g(X_T) q_T \Bigr]
\label{eq:costsoftkilling28/04}
\end{equation}
with $q_t := e^{-A_t}/\E [e^{-A_t}]$. As explained in \cite{CLL23} we recover the \textit{hard killing} regime by taking $V$ as in \eqref{eq:V0ouinfty28/04} in which case the cost is, at least formally, given by 
$$ \int_{t_0}^T \frac{1}{2} \E \bigl[ |\alpha_t(X_t)|^2 \big| \tau_D > t ] dt + \E \bigl[ g(X_T) | \tau_D >T], $$
where $\tau_D$ is the first exit time of the process $(X_t)_{t \geq t_0}$ from the domain $D$.

From dynamic programming, we can infer that $\mathcal{V}$ is a viscosity solution to the following Hamilton-Jacobi equation, set in $[0,T] \times \mathcal{P}(\T^d)$
\begin{equation}
    \left \{
    \begin{array}{ll}
\displaystyle    -\partial_t \mathcal{V}(t,\mu) - \frac{1}{2}\int_{\T^d} \Delta_x \frac{\delta \mathcal{V}}{\delta \mu}(t,\mu,x) d\mu(x) + \displaystyle \frac{1}{2} \int_{\T^d} \bigl| \nabla_x \frac{\delta \mathcal{V}}{\delta \mu}(t,\mu,x) |^2 d\mu(x) \\
 \displaystyle  \hspace{100pt} + \int_{\T^d}  (V(x) - \lg V;\mu \rg) \frac{\delta \mathcal{V}}{\delta \mu}(t,\mu,x) d\mu(x) =0, \quad \mbox{ in } [0,T] \times \mathcal{P}(\T^d), \\
\mathcal{V}(T, \mu) = \int_{\T^d} g(x) d\mu(x), \qquad \mu\in\mathcal{P}(\T^d), 
    \end{array}
    \right.
\label{eq:MasterHJB16/04Intro}
\end{equation}
where the derivative with respect to the probability measure argument is the so-called \emph{linear } or \emph{flat} derivative. See for example \cite{CarmonaDelarue_book_I}. In the equation above we used the convenient notation
$$ \langle V; \mu \rangle := \int_{\T^d} V(x) d\mu(x) $$
for $\mu \in \mathcal{P}(\T^d)$ to avoid double integrals and lighten the notations. 

\vskip 4pt
Our main convergence result is the following:

\begin{thm}
\label{thm:main28/04}
    Assume that $V: \T^d \rightarrow \R^+$ and $g: \T^d \rightarrow \R$ belong to $\mathcal{C}^k(\T^d)$ for some $k > d/2 +2$. There is a constant $C$ depending on $T$, $\norm{V}_{\mathcal{C}^k(\T^d)}$ and $\norm{g}_{\mathcal{C}^k(\T^d)}$ such that, for every $(t, \bx, \ba) \in [0,T] \times (\T^d)^N \times \R^N$,
$$ 
\bigl| v^N(t,\bx, \ba) -  \mathcal{V}(t, \mu_{\bx,\ba}^N) \bigr| \leq C \sqrt{ \sum_{i=1}^N \Bigl( \frac{1}{n^{i}[\ba]} \Bigr)^2    }
$$
where the weights $n^{i}[\ba]$ were defined in \eqref{fo:weights}, and where the \emph{weighted empirical measures} $\mu^N_{\bx,\ba}$ are defined as
\begin{equation}
    \label{fo:empirical}
\mu_{\bx,\ba}^N := \sum_{i=1}^N \frac{1}{n^{i}[\ba]} \delta_{x^{i}}.
\end{equation}
\end{thm}

\noindent
Theorem \ref{thm:main28/04} is a direct consequence of Propositions \ref{prop:hardineq19/04} and \ref{prop:easyineq19/04} proved in Sections \ref{sec:hardineq} and \ref{sec:easyineq}.

An immediate consequence is that, for every $R >0$ there is $C_R >0$ depending on $R$, $\norm{g}_{\mathcal{C}^k(\T^d)}$, $\norm{V}_{\mathcal{C}^k(\T^d)}$ and $T$ such that, for every $(t,\bx,\ba) \in [0,T] \times (\T^d)^N \times \R^N$ such that $|a|^{i} \leq R$ for all $1 \leq i \leq N$ it holds
$$ 
\bigl| v^N(t,\bx, \ba) -  \mathcal{V}(t, \mu_{\bx,\ba}^N) \bigr| \leq \frac{C_R}{\sqrt{N}}. 
$$
Theorem \ref{thm:main28/04} opens several perspectives. As alluded to earlier, the main one is to understand the regime of \textit{hard killing}. This could be tackled directly by setting $\frac{1}{N_t^{i}} = \frac{1}{ N_t} \mathbf{1}_{\tau^{i}_{D} >t} $ in the cost functionnal \eqref{eq:defncostJN228/04} or by  trying to prove the convergence of $v^N$ and $\mathcal{V}$ separately when the potential $V$ becomes singular as in \eqref{eq:V0ouinfty28/04}. In any case, finding estimates that guarantee that there are \textit{enough} particles remaining at time $T$ toemulate some form of mean-field limiting regime seems to be a tough challenge. To wit, in the control formulation, it could be beneficial to force all but one particle to exit the domain, leaving the surviving one at a location where the final cost $g$ is small.

\section{Analysis of the Limit value function}

\label{sec:AnalysisVF}

In this section we investigate the regularity properties of the value function $\mathcal{V} : [0,T] \times \mathcal{P}(\T^d) \rightarrow \R$ defined, for all $(t_0, \mu_0) \in [0,T] \times \mathcal{P}(\T^d)$ in \eqref{def:limitvfV}. The main result is Proposition \ref{prop:lipandSCV16/004} where we prove that $\mathcal{V}$ is Lipschitz and semi-concave with respect to appropriate  metrics. As highlighted in \cite{ddj2023}, this regularity is key to obtain sharp rates of convergence for the mean-field approximation. Compared to the results of \cite{ddj2023}, we need to address additional difficulties to obtain the regularity of the value function. In particular, the stability properties of the Fokker-Planck equation, see Lemma \ref{lem:stability15/04} below, are more delicate because of the non local term $\lg V;\mu \rg$ that creates a non-linearity in the equation.

The starting point is the following set of optimality conditions for the control problem \eqref{def:limitvfV} which extends  the analysis of \cite[Section 4.3.3]{CLL23}. The proof is postponed to Section \ref{sec:FOC}.

\begin{thm}
\label{thm:OptimalityConditons}
    Let us assume that $V : \T^d \rightarrow \R^+$ and $g: \T^d \rightarrow \R$ belong to $\mathcal{C}^k(\T^d)$ for some $k \geq 2$. For every $(t_0,\mu_0) \in [0,T] \times \mathcal{P}(\T^d)$, optimal solutions for \eqref{def:limitvfV} exist and satisfy $ \alpha = -\nabla u$ for a solution $(u,\mu)$ of the coupled PDE system
\begin{equation}
\left \{
    \begin{array}{ll}
    -\partial_t u_t - \frac{1}{2} \Delta u_t + \frac{1}{2} |\nabla u_t|^2 + (V- \langle V; \mu_t \rangle) u_t  - V \langle u_t; \mu_t \rangle = 0 \quad \mbox{ in } (t_0,T) \times \T^d, \\
    \partial \mu_t - \frac{1}{2} \Delta \mu_t - \div( \nabla u_t \mu_t) + (V- \langle V;\mu_t \rangle)\mu_t = 0 \quad \mbox{ in } (t_0,T) \times \T^d, \\
    \mu(t_0) = \mu_0, \quad u_T = g \quad \quad \mbox{ in } \T^d.
    \end{array}
\right.
\end{equation}
The adjoint variable $u$ belongs to $\mathcal{C}^{1,2}([0,T] \times\T^d)$ (one  continuous time derivative and two continuous space derivatives) and $x \mapsto u(t,x)$ belongs to $\mathcal{C}^k(\T^d)$ for all $t \in [t_0,T]$ and to $\mathcal{C}^{k+1}(\T^d)$ for all $t \in [t_0,T)$. More precisely, there is a non-decreasing function $\Lambda : \R^+ \rightarrow \R^+$, independent of $(t_0,\mu_0,V,g,\alpha,T)$ such that 
\begin{equation} 
\sup_{t \in [t_0,T]} \norm{ u_t}_{\mathcal{C}^k(\T^d)} + \sup_{t \in [t_0,T)} \sqrt{T-t} \norm{ u_t}_{\mathcal{C}^{k+1}(\T^d)} \leq \Lambda \bigl(T  + \norm{V}_{\mathcal{C}^k(\T^d)} + \norm{g }_{\mathcal{C}^k(\T^d)} \bigr).
\label{eq:estimateforu15/04}
\end{equation}
In particular, any optimal control satisfies
\begin{equation} 
\sup_{t \in [t_0,T]} \norm{ \alpha_t}_{\mathcal{C}^{k-1}(\T^d)}  + \sup_{t \in [t_0,T)} \sqrt{T-t} \norm{ \alpha_t}_{\mathcal{C}^k(\T^d)}  \leq \Lambda \bigl(T  + \norm{V}_{\mathcal{C}^k(\T^d)} + \norm{g }_{\mathcal{C}^k(\T^d)} \bigr).  
\label{eq:estimateforalpha15/04}
\end{equation}

\end{thm}

We need some preliminary stability results for the Fokker-Planck equation \eqref{eq:conditionalFPE} with respect to the initial condition. 
Recall that, for $\mu^1, \mu^2 \in \mathcal{P}(\T^d)$, 
$$ \norm{ \mu^1 -\mu^2}_{\mathcal{C}^{-k}(\T^d)} := \sup_{ \varphi \in \mathcal{C}^k(\T^d), \norm{\varphi}_{\mathcal{C}^k(\T^d)} \leq 1} \int_{\T^d} \varphi(x) d(\mu^1-\mu^2)(x).$$

\begin{lem}
\label{lem:stability15/04}
Take $\alpha \in \mathcal{C}([t_0,T] \times \T^d; \R^d)$ with 
$$\sup_{t \in [t_0,T]} \norm{ \alpha_t}_{\mathcal{C}^{k-1}(\T^d)} +\sup_{t \in [t_0,T)} \sqrt{T-t} \norm{ \alpha_t}_{\mathcal{C}^k(\T^d)}  < +\infty.$$ 

$i)$   Assume that $\mu^{i}$, $i=1,2$ are two solutions to
    $$ \partial_t \mu_t^{i} - \frac{1}{2}\Delta \mu_t^{i} + \div(\alpha_t \mu_t^{i}) + (V - \langle V;\mu_t^{i} \rangle) \mu_t^{i} = 0, \quad \mu^{i}_{t_0} = \mu^{i}_0. $$
Then there is a non-decreasing function $\Lambda : \R^+ \rightarrow \R^+$ independent of $(t_0,T,k,\alpha, V,\mu^1_0, \mu^2_0)$ such that, 
$$ 
\sup_{t \in [t_0,T]} \norm{ \mu_t^1 - \mu_t^2 }_{\mathcal{C}^{-k}(\T^d)} \leq \Lambda \Bigl( T+ k + \sup_{t \in [t_0,T]} \norm{ \alpha_t}_{\mathcal{C}^{k-1}(\T^d)}  + \norm{ V}_{\mathcal{C}^{k-1}(\T^d)} \Bigr)  \norm{ \mu^1_{t_0} - \mu^2_{t_0}}_{\mathcal{C}^{-k}(\T^d)}.  
$$
$ii)$ Now, for $\lambda \in (0,1)$, define $\mu_{t_0}^{\lambda} := (1- \lambda) \mu_{t_0}^1 + \lambda \mu_{t_0}^2$ and let $(\mu_t^{\lambda})_{t \in [t_0,T]}$ be the solution to the Fokker-Planck equation starting from this initial position with the same control $\alpha$. Then, for a possibly new choice of $\Lambda$ satisfying the same properties,
\begin{align*}
\sup_{t \in [t_0,T]} \norm{ (1-\lambda) \mu_t^1 + \lambda \mu_t^2 - \mu_t^{\lambda}}_{\mathcal{C}^{-k}(\T^d)} &\leq 
 \Lambda \Bigl( T+ k + \sup_{t \in [t_0,T]} \norm{ \alpha_t}_{\mathcal{C}^{k-1}(\T^d)}  + \norm{ V}_{\mathcal{C}^{k}(\T^d)} \Bigr) \\
& \times \frac{\lambda (1-\lambda)}{2} \norm{ \mu_{t_0}^1 - \mu_{t_0}^2}^2_{\mathcal{C}^{-k}(\T^d)}.
\end{align*}

\end{lem}

\begin{proof}
 $i)$   We take $\varphi \in \mathcal{C}^k(\T^d)$ and we introduce, for some $t_1 \in (t_0,T]$, $\phi : [t_0,t_1] \times \R^d \rightarrow \R$ the solution to the linear backward equation
\begin{equation} 
-\partial_t \phi_t - \frac{1}{2}\Delta \phi_t - \alpha_t \cdot \nabla \phi_t + (V-\langle V;\mu_t^1 \rangle) \phi_t - V \langle \phi_t; \mu_t^1 \rangle = 0, \quad \phi_{t_1} = \varphi.
\label{eq:linearbacwardequation}
\end{equation}
Well posedness for this equation is discussed in Lemma \ref{lem:estimatelinbackwardeq15/04} below. By duality we find, for $i=1,2$, 
$$ \int_{\T^d} \varphi(x) d\mu_{t_1}^{i}(x) = \int_{\T^d} \phi_{t_0}(x)d \mu_{t_0}^{i}(x) + \int_{t_0}^{t_1}( \langle V; \mu_t^{i} - \mu_t^1 \rangle \langle \phi_t; \mu_t^{i} \rangle - \langle \phi_t ; \mu_t^1 \rangle \langle V, \mu_t^{i} \rangle ) dt,   $$
and therefore,
$$ \int_{\T^d} \varphi(x) d(\mu^1_{t_1} - \mu^2_{t_1})(x) = \int_{\T^d} \phi_{t_0}(x)d (\mu^1_{t_0} - \mu^2_{t_0})(x) - \int_{t_0}^{t_1} \langle V; \mu_s^1 - \mu_s^2 \rangle \langle \phi_s; \mu_s^1 - \mu_s^2 \rangle ds. $$
Using the estimate on $\norm{ \phi_{t_0}}_{\mathcal{C}^k(\T^d)}$ from Lemma \ref{lem:estimatelinbackwardeq15/04} below, as well as the boundness of $V$ we find a non deacreasing function $\Lambda$ as in the statement and from which we omit the argument, such that, 
$$ \int_{\T^d} \varphi(x) d(\mu^1_{t_1} - \mu^2_{t_1})(x) \leq \Lambda  \norm{ \varphi}_{\mathcal{C}^{k}(\T^d)}  \Bigl( \norm{ \mu^1_{t_0} - \mu^2_{t_0}}_{\mathcal{C}^{-k}(\T^d)} + \int_{t_0}^{t_1} \norm{ \mu^1_s - \mu^2_s}_{\mathcal{C}^{-k}} ds \Bigr).  $$
Taking the supremum over the functions $\varphi \in \mathcal{C}^k$ with $\norm{ \varphi}_{\mathcal{C}^k} \leq 1$ and then applying Grönwall's Lemma we find
$$ \sup_{t_1 \in [t_0,T]} \norm{ \mu^1_{t_1} - \mu^2_{t_1}}_{\mathcal{C}^{-k}(\T^d)} \leq \Lambda \norm{\mu^1_{t_0} - \mu^2_{t_0}}_{\mathcal{C}^{-k}(\T^d)}, $$
which concludes the proof of first part of the lemma.

$ii)$ For the second part we argue again by duality. For some $t_1 \in (t_0,T]$ and some $\varphi \in \mathcal{C}^k(\T^d)$ with $\norm{\varphi}_{\mathcal{C}^k(\T^d)} \leq 1$ we let $\phi : [t_0,t_1] \times \T^d \rightarrow \R$ be the solution to the backward equation \eqref{eq:linearbacwardequation}. We compute, by duality,
\begin{align*}
    \int_{\T^d} & \varphi(x) d ( (1-\lambda) \mu_{t_1}^1 + \lambda \mu_{t_1}^2 - \mu_{t_1}^{\lambda})(x) = \int_{\T^d} \phi_{t_0}(x) d ( (1-\lambda) \mu_{t_0}^1 + \lambda \mu_{t_0}^2 - \mu_{t_0}^{\lambda})(x) \\
    & -(1-\lambda) \int_{t_0}^{t_1} \langle \phi_t ; \mu_t^1\rangle \langle V; \mu_t^1 \rangle dt  + \lambda \int_{t_0}^{t_1} \langle V; \mu_t^2 - \mu_t^1 \rangle \langle \phi_t ; \mu_t^2 \rangle dt - \lambda \int_{t_0}^{t_1} \langle \phi_t, \mu_t^1 \rangle \langle V; \mu_t^2 \rangle dt \\
    &-\int_{t_0}^{t_1}( \langle V; \mu_t^{\lambda} - \mu_t^1 \rangle \langle \phi_t; \mu_t^{\lambda} \rangle - \langle \phi_t ; \mu_t^1 \rangle \langle V, \mu_t^{\lambda} \rangle ) dt.
\end{align*}
The first term cancels out by definition of $\mu_{t_0}^{\lambda}$ and then tedious rearranging leads to 
\begin{align*}
     \int_{\T^d} \varphi(x) d ( (1-\lambda) \mu_{t_1}^1 &+ \lambda \mu_{t_1}^2 - \mu_{t_1}^{\lambda})(x) = \int_{t_0}^{t_1} \lg \phi_t;\mu_t^1 \rg \lg V; \mu_t^{\lambda} - (1-\lambda) \mu_t^1 - \lambda \mu_t^2 \rg dt \\
     & - \int_{t_0}^{t_1} \lg \phi_t ; \mu_t^{\lambda} - (1-\lambda) \mu_t^1 - \lambda \mu_t^2 \rg \lg V; \lambda(\mu_t^2 - \mu_t^1) \rg dt \\
    &+ \lambda (1-\lambda) \int_{t_0}^{t_1} \lg \phi_t ; \mu_t^2- \mu_t^1 \rg \lg V; \mu_t^2 - \mu_t^1 \rg dt. 
\end{align*}
Using the $\mathcal{C}^k$ regularity of $V$ and $\phi$ we get 
\begin{align*}
     \int_{\T^d} \varphi(x) d ( (1-\lambda) \mu_{t_1}^1 &+ \lambda \mu_{t_1}^2 - \mu_{t_1}^{\lambda})(x) \leq  \int_{t_0}^{t_1} \norm{\phi_t}_{\mathcal{C}^0(\T^d)} \norm{V}_{\mathcal{C}^k(\T^d)} \norm{ \mu_t^{\lambda} - (1-\lambda) \mu_t^1 - \lambda \mu_t^2}_{\mathcal{C}^{-k}(\T^d)} dt \\
     & + \int_{t_0}^{t_1}  \norm{\phi_t}_{\mathcal{C}^k(\T^d)} \norm{ \mu_t^{\lambda} - (1-\lambda) \mu_t^1 - \lambda \mu_t^2}_{\mathcal{C}^{-k}(\T^d)}   \times 2 \lambda \norm{ V}_{\mathcal{C}^0(\T^d)}  \\
    &+ \lambda (1-\lambda) \int_{t_0}^{t_1}  \norm{\phi_t}_{\mathcal{C}^k(\T^d)} \norm{V}_{\mathcal{C}^k(\T^d)} \norm{ \mu_t^2- \mu_t^1}^2_{\mathcal{C}^{-k}(\T^d)} dt. 
\end{align*}
Now we use step $i)$ and the estimate of Lemma \ref{lem:estimatelinbackwardeq15/04} to find $\Lambda$ as in the statement such that 
\begin{align*}  
\int_{\T^d} \varphi(x) d ( (1-\lambda) \mu_{t_1}^1 + \lambda \mu_{t_1}^2 - \mu_{t_1}^{\lambda})(x) &\leq \Lambda \norm{\varphi}_{\mathcal{C}^k(\T^d)} \Bigl( \int_{t_0}^{t_1} \norm{ \mu_t^{\lambda} - (1-\lambda) \mu_t^1 - \lambda \mu_t^2}_{\mathcal{C}^{-k}(\T^d)} dt \\
&+ \frac{\lambda (1-\lambda)}{2}  \sup_{t \in [t_0,T]} \norm{ \mu_t^1 - \mu_t^2 }^2_{\mathcal{C}^{-k}(\T^d)} \Bigr).
\end{align*}
Taking the supremum over $\varphi$ with $\norm{\varphi}_{\mathcal{C}^k(\T^d)} \leq 1$ we infer
\begin{align*}
\norm{ (1-\lambda) \mu_{t_1}^1 + \lambda \mu_{t_1}^2 - \mu_{t_1}^{\lambda} }_{\mathcal{C}^{-k}(\T^d)} & \leq  \Lambda \int_{t_0}^{t_1} \norm{ (1-\lambda) \mu_t^1 + \lambda \mu_t^2 - \mu_t^{\lambda} }_{\mathcal{C}^{-k}(\T^d)} dt  \\
&+ \frac{\Lambda}{2} \lambda (1-\lambda) \sup_{t \in [t_0,T]} \norm{ \mu_t^1 - \mu_t^2 }^2_{\mathcal{C}^{-k}(\T^d)}  
\end{align*}
and we conclude, using Grönwall's lemma and the  part $i)$ of the Lemma that, for a possibly new function $\Lambda$,
$$ \sup_{t \in [t_0,T]} \norm{ (1-\lambda) \mu_t^1 + \lambda \mu_t^2 - \mu_t^{\lambda} }_{\mathcal{C}^{-k}(\T^d)} \leq \Lambda \frac{\lambda (1-\lambda)}{2}  \norm{ \mu_{t_0}^1 - \mu_{t_0}^2 }_{\mathcal{C}^{-k}(\T^d)}^2.  $$
\end{proof}

In the proof above we used the folowing estimate.
\begin{lem}
\label{lem:estimatelinbackwardeq15/04}
Take $t_1 \in [t_0,T]$, $\mu \in \mathcal{C}([t_0,t_1], \mathcal{P}(\T^d))$ and $\varphi \in \mathcal{C}^k(\T^d)$ for some $k \geq 1$. Let $\phi : [t_0,t_1] \times \T^d \rightarrow \R$ be the classical solution to 
    \begin{equation} 
-\partial_t \phi_t - \frac{1}{2}\Delta \phi_t - \alpha_t \cdot \nabla \phi_t + (V-\langle V;\mu_t^1 \rangle) \phi_t - V \langle \phi_t; \mu_t^1 \rangle = 0, \quad \phi_{t_1} = \varphi.
\end{equation}
Then, there is a non-decreasing function $\Lambda: \R^+ \rightarrow \R^+$, independent of $(t_0,t_1,T,\mu,V,\varphi,k)$ such that 
$$ \sup_{t \in [t_0,t_1]} \norm{ \phi_t}_{\mathcal{C}^k(\T^d)} \leq \Lambda \bigl( T + k + \sup_{t \in [t_0,t_1]} \norm{ \alpha_t}_{\mathcal{C}^{k-1}(\T^d)} + \norm{V}_{\mathcal{C}^{k-1}(\T^d)} \bigr) \norm{ \varphi}_{\mathcal{C}^k(\T^d)}. $$
\end{lem}

\begin{proof}
Existence of a solution is given by \cite[Lemma 8]{CLL23} thanks to a fixed point argument. We  claim that 
 $$ \sup_{t \in [t_0,t_1]} \norm{ \phi_t}_{\mathcal{C}^0(\T^d)} \leq \Lambda \bigl( T + \norm{V}_{\mathcal{C}^0(\T^d)} \bigr) \norm{\varphi}_{\mathcal{C}^0(\T^d)} $$
for a non-decreasing function $\Lambda$ as in the statement. This follows, for instance, from the same argument leading to the $\mathcal{C}^0$ estimate for $u^R$ in Lemma \ref{lem:AprioriEstimatesHJB14/04}. We then proceed by induction, following the proof of  \cite[Lemma A.2]{ddj2023}. We introduce the heat kernel  $(P_t)_{t \geq 0}$ so that $\phi_t$ is given for all $t\in [t_0,t_1]$ by Duhamel's formula 
$$ \phi_t = P_{t_1 -t} \varphi + \int_{t}^{t_1} P_{s-t} \bigl \{ \alpha_s \cdot \nabla \phi_s - (V - \langle V ; \mu_s \rangle ) \phi_s + V \langle \phi_s ; \mu_s \rangle \bigr \} ds. $$
Differentiating $k$ times on both side and using the classical estimates
$$\norm{ P_{s-t} f}_{\mathcal{C}^0(\T^d)} \leq \norm{f}_{\mathcal{C}^0(\T^d)}, \quad \quad  \norm{\nabla \bigl \{ P_{s-t} f \bigr \}}_{\mathcal{C}^0(\T^d)} \leq \frac{\norm{f}_{\mathcal{C}^{0}(\T^d)}}{\sqrt{s-t}} $$
we find
\begin{align} 
\norm{ \nabla^k \phi_t}_{\mathcal{C}^0(\T^d)} & \leq \norm{\nabla^k \varphi}_{\mathcal{C}^0(\T^d)} + \Lambda(k) \int_t^{t_1} \frac{\norm{\alpha_s}_{\mathcal{C}^{k-1}(\T^d)} \norm{ \phi_s}_{\mathcal{C}^k(\T^d)}}{\sqrt{s-t}} ds \\
&+ \Lambda(k) \int_t^{t_1} \norm{ \phi_s}_{\mathcal{C}^{k-1}(\T^d)} \frac{\norm{V}_{\mathcal{C}^{k-1}(\T^d)}}{\sqrt {s-t}}ds  + \int_t^{t_1} \norm{\phi_s}_{\mathcal{C}^0(\T^d)} \frac{\norm{ \nabla^{k-1}V}_{\mathcal{C}^{0}(\T^d)}}{\sqrt{s-t}}ds
\end{align}
for some $\Lambda(k)>0$ depending only on $k$, and we conclude by induction, using a variant of Grönwall's Lemma, see \cite[Lemma A.1]{ddj2023}.
\end{proof}

We go on with the first regularity properties of the value function $\mathcal{V}$. The result involves the metric $\norm{ \cdot }_{H^{-k}(\T^d)}$ inherited by duality with the Hilbert space $H^k(\T^d)$ of functions on $\T^d$ with derivatives up to order $k$  in $L^2(\T^d)$ (see the beginning Section \ref{sec:supconvolandproperties16/04} for details regarding the space $H^{-k}(\T^d)$).  

\begin{prop}
\label{prop:lipandSCV16/004}
    Assume that $V$ and $g$ belong to $\mathcal{C}^k(\T^d)$ for some $k \geq 2$. Then the value function $\mathcal{V}$ is Lipschitz and semi-concave in the measure argument with respect to the norm of $\mathcal{C}^{-k}(\T^d)$. More precisely, there is a non decreasing function $\Lambda : \R^+ \rightarrow \R^+$ independent of $(T,g,V)$ such that, for all $t_0 \in [0,T]$ and $\mu^1, \mu^2 \in \mathcal{P}(\T^d)$ we have:
\begin{equation} 
|\mathcal{V}(t_0,\mu^2) - \mathcal{V}(t_0,\mu^1)| \leq \Lambda \Bigl(T+ k + \norm{ V}_{\mathcal{C}^{k}(\T^d)} + \norm{g}_{\mathcal{C}^k(\T^d)} \Bigr) \norm{ \mu^2 - \mu^1 }_{\mathcal{C}^{-k}(\T^d)}, 
\label{eq:Lip/C-k16/03}
\end{equation}
and for all $\lambda \in (0,1)$
\begin{align} 
\notag \mathcal{V}(t_0, (1-\lambda) \mu^1 + \lambda \mu^2 ) &\geq (1-\lambda) \mathcal{V}(t_0,\mu^1) + \lambda \mathcal{V}(t_0,\mu^2) \\
&- \Lambda \Bigl(T+ k + \norm{ V}_{\mathcal{C}^{k}(\T^d)} + \norm{g}_{\mathcal{C}^k(\T^d)} \Bigr) \times \frac{\lambda (1-\lambda)}{2} \norm{ \mu^2 - \mu^1}^2_{\mathcal{C}^{-k}(\T^d)}. \label{eq:semiconcavC-k16/04}
\end{align}
Moreover, estimates \eqref{eq:Lip/C-k16/03} and \eqref{eq:semiconcavC-k16/04} still hold true, possibly for a new choice of $\Lambda$ but depending on the same parameters,  with $\norm{\mu^2 - \mu^1}_{\mathcal{C}^{-k}(\T^d)}$ replaced by $\norm{\mu^2 - \mu^1}_{H^{-k}(\T^d)}.$
\end{prop}

\begin{proof}
Throughout the proof, the non-decreasing function $\Lambda$ is allowed to change from line to line. We also omit its argument to lighten the notation. We start with the Lipschitz regularity. We fix some initial time $t_0 \in [0,T]$ and we take two measures $\mu_{t_0}^1,\mu_{t_0}^2$ as well as an optimal control $\alpha$ for $\mathcal{V}(t_0, \mu^1_{t_0})$. Thanks to Theorem \ref{thm:OptimalityConditons} we know that 
\begin{equation}
\sup_{t \in [t_0,T]} \norm{ \alpha_t}_{\mathcal{C}^{k-1}(\T^d)} 
+ \sup_{t \in [t_0,T]} \sqrt{T-t} \norm{ \alpha_t}_{\mathcal{C}^k(\T^d)} \leq \Lambda.
\label{eq:27/04:21:04}
\end{equation}
For all $t \in [t_0,T]$, expanding the $\mathcal{C}^k$ norm of the square $|\alpha_t|^2$ we find that, for some $C_{k-1}>0$ depending only on the $\mathcal{C}^{k-1}$ norm of $\alpha_t$,
$$ \norm{ |\alpha_t|^2} \leq C_{k-1} (1+ \norm{ \alpha_t}_{\mathcal{C}^k(\T^d)}).$$
Therefore we deduce from \eqref{eq:27/04:21:04} that
\begin{equation}
\sup_{t \in [t_0,T]} \norm{ |\alpha_t|^2}_{\mathcal{C}^{k-1}(\T^d)} 
+ \sup_{t \in [t_0,T]} \sqrt{T-t} \norm{ |\alpha_t|^2}_{\mathcal{C}^k(\T^d)} \leq \Lambda.
\label{eq:27/04:21:07}
\end{equation}
We let $(\mu^{i}_t)_{t \in [t_0,T]}$, $i=1,2$ be the measure flows obtained by using the same control $\alpha$ for both initial conditions $(t_0,\mu_{t_0}^1)$ and $(t_0, \mu_{t_0}^1)$. Thanks to Lemma \ref{lem:stability15/04} we know that 
$$ 
\sup_{t \in [t_0,T]} \norm{ \mu_t^1 -\mu_t^2}_{\mathcal{C}^{-k}(\T^d)} \leq \Lambda \norm{ \mu_{t_0}^1 - \mu_{t_0}^1}_{\mathcal{C}^{-k}(\T^d)}. 
$$
Since $\alpha$ is optimal for $(t_0,\mu_{t_0}^1)$ and admissible for $(t_0,\mu_{t_0}^2)$ we have 
\begin{align*}
    \mathcal{V}(t_0, \mu_{t_0}^2) - \mathcal{V}(t_0, \mu_{t_0}^1) & \leq \int_{t_0}^T \int_{\T^d} \frac{1}{2} |\alpha_t(x)|^2 d(\mu_t^2 - \mu_t^1)(x)dt + \int_{\T^d} g(x) d(\mu_T^2 - \mu_T^1)(x) \\
    &\leq \Lambda \sup_{t \in [t_0,T]}  \norm{ \mu^2_t - \mu^1_t}_{\mathcal{C}^{-k}(\T^d)} ,
\end{align*}
where we used the regularity of $\alpha$ through the bound \eqref{eq:27/04:21:07} as well as the $\mathcal{C}^k$ regularity of $g$. Applying part $i)$ of Lemma \ref{lem:stability15/04} we deduce that
$$ \mathcal{V}(t_0, \mu_{t_0}^2) - \mathcal{V}(t_0, \mu_{t_0}^1) \leq \Lambda \norm{ \mu_{t_0}^1 - \mu_{t_0}^2}_{\mathcal{C}^{-k}(\T^d)}.$$
Exchanging the roles of $\mu_{t_0}^1$ and $\mu_{t_0}^2$ we get the Lipschitz regularity of $\mathcal{V}$. 

We go on with the semi-concavity estimates. To this end we fix and initial time $t_0 \in [0,T)$, two initial measures $\mu_{t_0}^1$ and $\mu_{t_0}^2$. For all $\lambda \in [0,1]$ we let $\mu_{t_0}^{\lambda} := (1-\lambda) \mu_{t_0}^1 + \lambda \mu_{t_0}^2$. For $\lambda \in (0,1)$ fixed we take $\alpha^{\lambda}$ an optimal control for $\mathcal{V}(t_0, \mu_{t_0}^{\lambda})$ and we consider the curves $(\mu^1_t)_{t \in [t_0,T]}, (\mu^2_t)_{t \in [t_0,T]}$ and $(\mu^{\lambda}_t)_{t\in [t_0,T]}$ obtained by playing $\alpha^{\lambda}$ for each initial condition. By optimality of $\alpha^{\lambda}$ for $\mathcal{V}(t_0, \mu^{\lambda}_{t_0})$ we have
\begin{align*}
    (1-\lambda) &\mathcal{V}(t_0, \mu_{t_0}^1) + \lambda \mathcal{V}(t_0, \mu_{t_0}^2) - \mathcal{V}(t_0, \mu_{t_0}^{\lambda}) \\
    &\leq \int_{t_0}^T \int_{\T^d} \frac{1}{2} |\alpha^{\lambda}_t(x)|^2 d( (1-\lambda) \mu_t^1 + \lambda \mu_t^2 - \mu_t^{\lambda})(x) + \int_{\T^d} g(x)d ( (1-\lambda) \mu_T^1 + \lambda \mu_T^2 - \mu_T^{\lambda})(x) \\
    & \leq  \sup_{t \in [t_0,T]} \norm{(1-\lambda)\mu_t^1 + \lambda \mu_t^2 - \mu_t^{\lambda}}_{\mathcal{C}^{-k}(\T^d)} \Bigl( \int_{t_0}^T \frac{1}{2} \norm{ |\alpha_t^{\lambda} |^2}_{C^k(\T^d)} dt + \norm{ g}_{\mathcal{C}^k(\T^d)} \Bigr).
\end{align*}

We conclude using part $ii)$ of Lemma \ref{lem:stability15/04} and estimate \eqref{eq:27/04:21:07} (which holds equally well for $\alpha^{\lambda}$)  that,  
$$  (1-\lambda) \mathcal{V}(t_0, \mu_{t_0}^1) + \lambda \mathcal{V}(t_0, \mu_{t_0}^2) - \mathcal{V}(t_0, \mu_{t_0}^{\lambda}) \leq \Lambda \frac{\lambda (1-\lambda)}{2}  \norm{ \mu_{t_0}^1 - \mu_{t_0}^2}^2_{\mathcal{C}^{-k}(\T^d)}.$$
Finally, since we are working over the compact space $\T^d$, there is a constant $C_{d,k}$ such that $\norm{ \varphi}_{H^k(\T^d)} \leq C_{d,k} \norm{\varphi}_{\mathcal{C}^k(\T^d)} $ for all $ \varphi \in \mathcal{C}^k(\T^d)$ and therefore $\norm{ }_{\mathcal{C}^{-k}} \leq C_{d,k} \norm{ \cdot }_{H^{-k}(\T^d)}$ for the same constant $C_{d,k}$ and we can replace $\norm{ \mu^2 - \mu^1 }_{\mathcal{C}^{-k}(\T^d)}$ by $\norm{ \mu^2 - \mu^1}_{H^{-k}(\T^d)}$ in estimates  \eqref{eq:Lip/C-k16/03} and \eqref{eq:semiconcavC-k16/04}.
This concludes the proof of the Lemma.
\end{proof}

Before going on with the time regularity we recall that, by classical arguments,  $\mathcal{V}$ satisfies a dynamic programming principle.

\begin{prop}[Dynamic Programming Principle]
    For any $(t_0, \mu_0) \in [0,T] \times \mathcal{P}(\T^d)$ and any $t_1 \in (t_0,T]$ it holds
    $$\mathcal{V}(t_0,\mu_0) = \inf_{(\alpha_t, \mu_t)} \int_{t_0}^{t_1} \int_{\T^d} \frac{1}{2}|\alpha_t(x)|^2 d\mu_t(x)dt + \mathcal{V}(t_1, \mu_{t_1}), $$
where the infimum is taken over the couples $(\mu,\alpha)$ solutions to
$$ \partial_t \mu_t + \div( \alpha_t \mu_t ) - \frac{1}{2}\Delta \mu_t + (V - \lg V;\mu_t \rg)\mu_t = 0, \mbox{  in } (t_0,t_1) \times \T^d, \quad \mu_{t_0} = \mu_0 \quad \mbox{ in } \T^d,$$
with $\mu \in \mathcal{C}([t_0,t_1],\mathcal{P}(\T^d))$ and $\alpha \in L_{\mu_t \otimes dt}^{\infty}([t_0,t_1] \times \T^d; \R^d).$
\end{prop}

\begin{prop}
\label{prop:timeregV}
    Assume that $V$ and $g$ belong to $\mathcal{C}^k(\T^d)$ for some $k \geq 2$. Then there is a non decreasing function $\Lambda : \R^+ \rightarrow \R^+$ independent of $(T,g,V)$ such that, for all $t_1,t_2 \in [0,T]$, all $\mu^0 \in \mathcal{P}(\T^d)$
    $$ |\mathcal{V}(t_2,\mu^0) - \mathcal{V}(t_1,\mu^0)| \leq \Lambda \Bigl(T + \norm{ V}_{\mathcal{C}^{2}(\T^d)} + \norm{g}_{\mathcal{C}^2(\T^d)} \Bigr) |t_2 -t_1|.  $$
\end{prop}

\begin{proof}
    We fix some initial measure $\mu^0 \in \mathcal{P}(\T^d)$ and two times $t_1,t_2 \in [0,T]$ with $t_1 <t_2$. We let $(\alpha_t)_{t \geq t_1}$ be an optimal control for $\mathcal{V}(t_1,\mu_0)$ and we denote by $(\mu_t)_{t\geq t_1}$ the associated measure flow. By dynamic programming we have,
\begin{align*}
\mathcal{V}(t_1, \mu_0) - \mathcal{V}(t_2,\mu_0) &= \mathcal{V}(t_1, \mu_0) - \mathcal{V}(t_1,\mu_{t_2}) + \mathcal{V}(t_2, \mu_{t_2}) - \mathcal{V}(t_2, \mu_0) \\
&= \frac{1}{2}\int_{t_0}^{t_1} \int_{\T^d} |\alpha_t(x)|^2d\mu_t(x)dt + \mathcal{V}(t_2, \mu_{t_2}) - \mathcal{V}(t_2, \mu_0).
\end{align*}
Using the Lipschitz regularity in the measure variable of $\mathcal{V}$ as well as the boundness of $\alpha$ we infer that
\begin{equation} 
|\mathcal{V}(t_1, \mu_0) - \mathcal{V}(t_2,\mu_0) |\leq \Lambda |t_2 -t_1| + \Lambda \norm{ \mu_{t_2} - \mu_0}_{\mathcal{C}^{-2}(\T^d)}.
\label{eq:towardtimLip15/04}
\end{equation}
To estimate the last term we proceed as follows. We take $\phi \in \mathcal{C}^2(\T^d)$ and we compute, by duality,
\begin{align*}
    \int_{\R^d} \phi(x) d(\mu_{t_2} - \mu_0)(x) &= \int_{t_1}^{t_2} \int_{\T^d} \bigl[ \alpha_t(x) \cdot \nabla \phi(x) + \frac{1}{2}\Delta \phi(x) - (V- \lg V;\mu_t \rg ) \phi(x) \bigr] d\mu_t(x)dt \\
    &\leq \Lambda |t_2 -t_1| \norm{ \phi }_{\mathcal{C}^2(\T^d)} 
\end{align*}
and we take the supremum over $\phi \in \mathcal{C}^2(\T^d)$ with $\norm{ \phi}_{\mathcal{C}^2(\T^d)} \leq 1$ to deduce  that
$$ \norm{ \mu_{t_2} - \mu_0 }_{\mathcal{C}^{-2}(\T^d)} \leq \Lambda |t_2 -t_1|.$$
Combined with \eqref{eq:towardtimLip15/04} this concludes the proof of the time Lipchitz regularity of $\mathcal{V}$.
\end{proof}

From Dynamic Programming, we can infer that $\mathcal{V}$ is a viscosity solution to the following Hamilton-Jacobi equation, set in $[0,T] \times \mathcal{P}(\T^d)$
\begin{equation}
    \left \{
    \begin{array}{ll}
\displaystyle    -\partial_t \mathcal{V} - \frac{1}{2}\int_{\T^d} \Delta_x \frac{\delta \mathcal{V}}{\delta \mu}(t,\mu,x) d\mu(x) + \displaystyle \frac{1}{2} \int_{\T^d} \bigl| \nabla_x \frac{\delta \mathcal{V}}{\delta \mu}(t,\mu,x) |^2 d\mu(x) \\
 \displaystyle  \hspace{100pt} + \int_{\T^d}  (V(x) - \lg V;\mu \rg) \frac{\delta \mathcal{V}}{\delta \mu}(t,\mu,x) d\mu(x) =0, \quad \mbox{ in } [0,T] \times \mathcal{P}(\T^d), \\
\mathcal{V}(T, \mu) = \int_{\T^d} g(x) d\mu(x), \mbox{ in } \mathcal{P}(\T^d). 
    \end{array}
    \right.
\label{eq:MasterHJB16/04}
\end{equation}

In the equation above we use the linear derivative $\frac{\delta }{\delta \mu}$ with respect to the measure. Following \cite{CarmonaDelarue_book_I} we say that $U = \mathcal{P}(\T^d) \rightarrow \R$ is $\mathcal{C}^1$ over $\mathcal{P}(\T^d)$, in notation $U \in \mathcal{C}^1(\mathcal{P}(\T^d))$, if there is a jointly continuous map $\frac{\delta U}{\delta m} : \mathcal{P}(\T^d) \times \T^d \rightarrow \R$ such that, for all $m^1,m^2 \in \mathcal{P}(\T^d)$,
$$ U(m^2) - U(m^1) = \int_0^1 \int_{\T^d} \frac{\delta U}{\delta m} \bigl( (1-r) m^1 + rm^2,x \bigr) d (m^2 -m^1)(x)dr.$$
As such, the linear derivative would only be defined up to an additive constant, so we adopt the convention
\begin{equation} 
\int_{\T^d} \frac{\delta U}{\delta \mu}(\mu,x) dx = 0 \quad \forall \mu \in \mathcal{P}(\T^d).
\label{eq:normalizingconvention16/04}
\end{equation}

In the next result, we say that $\Phi : [0,T] \times \mathcal{P}(\T^d) \rightarrow \R$ is a smooth test function if  $\Phi$ belongs to $\mathcal{C}^1([0,T] \times \mathcal{P}(\T^d))$ and $x \mapsto \frac{\delta U}{\delta \mu}(t,\mu,x)$ is twice differentiable with jointly continuous derivatives.

\begin{lem}
\label{lem:subsolpropV16/04}
    The value function $\mathcal{V}$ is a viscosity sub-solution to the HJB equation \eqref{eq:MasterHJB16/04}. More precisely, for every smooth test function $\Phi :[0,T] \times \mathcal{P}(\T^d) \rightarrow \R$ such that $\mathcal{V} - \Phi$ has a maximum at $(t_0,\mu_0) \in [0,T] \times \mathcal{P}(\T^d)$ we have 
    \begin{align*} 
- \partial_t \Phi(t_0,\mu_0) -\frac{1}{2} \int_{\T^d} \Delta_x \frac{\delta \Phi}{\delta \mu}(t_0,\mu_0,x) &d\mu_0(x) + \frac{1}{2} \int_{\T^d} \bigl|\nabla_x \frac{\delta \Phi}{\delta \mu}(t,\mu_0,x) \bigr|^2 d\mu_0(x) \\
&+ \int_{\T^d} \bigl( V(x) - \lg V,\mu_0 \rg \bigr) \frac{\delta \Phi}{\delta \mu}(t_0, \mu_0,x) d\mu_0(x) \leq 0. 
\end{align*}
\end{lem}

We omit to state and prove the corresponding super-solution property that can be proved in an analogous way because it will not be needed to obtain the rate of convergence.

\begin{proof} 
    We take $(t_0,\mu_0) \in (0,T) \times \mathcal{P}(\T^d)$ and a smooth test function $\Phi : [0,T] \times \mathcal{P}(\T^d) \rightarrow \R $ such that $\mathcal{V} - \Phi$ has a maximum at $(t_0,\mu_0)$. Without loss of generality (up to adding a constant to $\Phi$) we can assume that this maximum is $0$. By dynamic programming we have, for all $t_1 \in (t_0,T]$,
\begin{align*}
    \Phi(t_0,\mu_0) = \mathcal{V}(t_0,\mu_0) &= \inf_{(\mu_t, \alpha_t)} \int_{t_0}^{t_1} \int_{\T^d} \frac{1}{2} |\alpha_t(x)|^2 d\mu_t(x)dt + \mathcal{V}(t_1, \mu_{t_1}) \\
    & \leq \inf_{(\mu_t, \alpha_t)} \int_{t_0}^{t_1} \int_{\T^d} \frac{1}{2} |\alpha_t(x)|^2 d\mu_t(x)dt + \Phi(t_1, \mu_{t_1}),
\end{align*}
where the infimum is taken over the couples $(\mu, \alpha)$ solution to 
$$ \partial_t \mu_t - \frac{1}{2}\Delta \mu_t + \div( \alpha_t \mu_t) + (V - \lg V; \mu_t \rg ) \mu_t = 0, \quad \mbox{ in } (t_0,t_1) \times \T^d, \quad \mu_{t_0} = \mu_0.$$

Using Theorem 5.99 in \cite{CarmonaDelarue_book_I} with a minor modification to account for the zero-th order term in the Fokker-Planck equation, we expand $\Phi$ along the flow of $\mu_t$ to obtain 
\begin{align*} 
\Phi(t_1,\mu_{t_1}) = \Phi(t_0, \mu_0) + \int_{t_0}^{t_1} \int_{\T^d} &\bigl[ \partial_t \Phi(t, \mu_t) + \alpha_t(x) \cdot \nabla_x\frac{\delta \Phi}{\delta \mu}(t,\mu_t,x) + \frac{1}{2}\Delta_x \frac{\delta \Phi}{\delta \mu}(t,\mu_t,x) \\
&- \bigl( V(x) - \lg V; \mu_t \rg \bigr) \frac{\delta \Phi}{\delta \mu}(t,\mu_t,x) \bigl] d\mu_t(x)dt.
\end{align*}
Taking $\alpha_t(x) = - \nabla_x \frac{\delta \Phi}{\delta \mu}(t,\mu_0,x)$ for all $t \in [t_0,t_1]$ and all $x \in \T^d$, we get
\begin{align*}
    0& \leq \int_{t_0}^{t_1} \int_{\T^d}  \bigl[ \partial_t \Phi(t,\mu_t) +  \frac{1}{2}\Delta_x \frac{\delta \Phi}{\delta \mu}(t,\mu_t,x) - \bigl( V(x) - \lg V; \mu_t \rg \bigr) \frac{\delta \Phi}{\delta \mu}(t,\mu_t,x)   \\
    & + \frac{1}{2} \bigl| \nabla_x\frac{\delta \Phi}{\delta \mu}(t,\mu_0,x) \bigr|^2  - \nabla_x \frac{\delta \Phi}{\delta \mu}(t,\mu_0,x) \cdot \nabla_x \frac{\delta \Phi}{\delta \mu}(t,\mu_t,x)  \bigr] d\mu_t(x)dt.
\end{align*}
Dividing by $t_1-t_0$ and letting $t_1 \rightarrow t_0$ we obtain
\begin{align*} 
- \partial_t \Phi(t_0,\mu_0) -\frac{1}{2} \int_{\T^d} \Delta_x \frac{\delta \Phi}{\delta \mu}(t_0,\mu_0,x) &d\mu_0(x) + \frac{1}{2} \int_{\T^d} \bigl|\nabla_x \frac{\delta \Phi}{\delta \mu}(t,\mu_0,x) \bigr|^2 d\mu_0(x) \\
&+ \int_{\T^d} \bigl( V(x) - \lg V;\mu_0 \rg \bigr) \frac{\delta \Phi }{\delta \mu}(t_0, \mu_0,x) d\mu_0(x) \leq 0, 
\end{align*}
which is the desired sub-solution property.
\end{proof}

\section{Rate of convergence: formal argument}

\label{sec:formalargumentprojection}

In this section we explain how to use the equation satisfied by $\mathcal{V}$ to infer quantitative information about the convergence of $v^N$. We assume that we are in an idealized setting where $\mathcal{V}$ is twice differentiable with bounded derivatives and the dynamic programming equation \eqref{eq:MasterHJB16/04} is satisfied in a strong sense. 
We introduce the function $\tilde{v}^N$ defined, for all $N \geq 1$ and for all $(t, \bx , \ba) \in [0,T] \times (\T^d)^N \times \R^N$ by
$$ 
\tilde{v}^N(t, \bx, \ba) = \mathcal{V} ( t, \mu^N_{\bx,\ba} ). 
$$
Recall that the empirical measures $\mu^N_{\bx,\ba}$ on $\T^d$ were defined in \eqref{fo:empirical} by:
$$
\mu^N_{\bx,\ba} = \sum_{i=1}^N \frac{1}{n^{i}[\ba]}\delta_{x^{i}},  \quad \quad n^{i}[\ba] = e^{a^{i}} \sum_{j=1}^N e^{-a^{j}} \quad 1 \leq i \leq N. 
$$
We can compute the derivatives of $\tilde{v}^N$ as follows
$$ \partial_t \tilde{v}^N(t, \bx, \ba) = \partial_t \mathcal{V}(t, \mu^N_{\bx , \ba}),$$
$$\partial_{x^{i}} \tilde{v}^N(t, \bx,\ba) = \frac{1}{n^{i}} \nabla_x \frac{\delta \mathcal{V}}{\delta \mu}(t, \mu^N_{\bx,\ba},x^{i}), $$
$$ \partial^2_{x^{i}x^{j}} \tilde{v}^N(t, \bx,\ba) = \delta_{ij} \frac{1}{n^{i}} \nabla^2_x \frac{\delta \mathcal{V}}{\delta \mu}(t, \mu^N_{\bx,\ba}, x^{i}) +  \frac{1}{n^{i}} \frac{1}{n^{j}} \nabla^2_{xy} \frac{\delta^2 \mathcal{V}}{\delta \mu^2}(t, \mu^N_{\bx,\ba}, x^{i}, x^{j}),$$
and finally (after some more or less tedious computations),
$$\partial_{a^{i}} \tilde{v}^N(t, \bx,\ba) =-\frac{1}{n^{i}} \Bigl(  \frac{\delta \mathcal{V}}{\delta \mu}(t, \mu^N_{\bx,\ba},x^{i}) - \int_{\T^d} \frac{\delta \mathcal{V}}{\delta \mu}(t, \mu^N_{\bx,\ba}, x) d\mu^N_{\bx,\ba}(x) \Bigr).$$
Evaluating the equation satisfied by $\mathcal{V}$ at $\mu_{\bx,\ba}^N$ and using the above relations we find that $\tilde{v}^N$ solves (still formally, assuming the $\tilde{v}^N$ is regular enough),
\begin{equation}
\left \{
\begin{array}{ll}
 \displaystyle    -\partial_t \tilde{v}^N - \frac{1}{2}\sum_{i=1}^N \Delta_{x^{i}} \tilde{v}^N - \sum_{i=1}^N V(x^{i}) \partial_{a^{i}} \tilde{v}^N + \sum_{i=1}^N \frac{n^{i}[\ba]}{2} |\nabla_{x^{i}} \tilde{v}^N|^2 \\
\displaystyle \hspace{120pt}= - \sum_{i=1}^N \bigl( \frac{1}{n^{i}[\ba]} \bigr)^2 \tr \nabla^2_{xy} \frac{\delta^2 \mathcal{V}}{\delta \mu^2}(t, \mu^N_{\bx,\ba},x^{i},x^{i})   , \quad \mbox{ in } [t_0,T]\times (\T^d)^N \times (\R)^N \\ 
\displaystyle    \tilde{v}^N(T, \bx, \ba) = \sum_{i=1}^N \frac{1}{n^{i}[\ba]} g(x^{i}), \quad \mbox{ in } (\T^d)^N \times (\R)^N. 
\end{array}
\right.
\end{equation}
We fix initial conditions $(t_0, \bx, \ba)$ and take an admissible control $(\alpha_t^1, \dots, \alpha_t^N)_{t \geq t_0}$ together with the corresponding trajectories $(\bX_t, \bA_t)_{t \geq t_0}$. Recalling the notations introduced in \eqref{fo:N-state_dynamics}, which are nothing more than the randomly controlled versions of  the deterministic notations \eqref{fo:weights} and \eqref{fo:empirical}, we introduce the notation
$$ 
\frac{1}{N_t^{i}}  := \frac{1}{n_i[\bA_t]} =\frac{e^{-A_t^{i}}}{\sum_{j=1}^N e^{-A_t^{j}}}, 
\qquad\text{and} \qquad 
\mu^N_t := \mu^N_{\bX_t,\bA_t}=\sum_{i=1}^N \frac{1}{N_t^{i}}\delta_{X_t^{i}} 
$$
and, by Itô's Lemma, we expand
\begin{align*}
    \mathcal{V}(T,\mu^N_T) &= \mathcal{V}(t_0,\mu^N_0) + \int_{t_0}^T \bigl \{ \partial_t \mathcal{V}(t,\mu^N_t) + \sum_{i=1}^N \frac{1}{N_t^{i}} \alpha_t^{i} \cdot \nabla_x \frac{\delta \mathcal{V}}{\delta \mu} (t, \mu^N_t , X_t^{i}) \bigr \} dt \\
    &+ \int_{t_0}^T \sum_{i=1}^N \frac{1}{N_t^{i}} V(X_t^{i}) \Bigl( \frac{\delta \mathcal{V}}{\delta \mu}(t, \mu^N_t, X_t^{i}) - \int_{\T^d} \frac{\delta \mathcal{V}}{\delta \mu}(t, \mu^N_t,x)d\mu^N_t(x) \Bigr) dt \\
    &+ \frac{1}{2} \int_{t_0}^T \sum_{i=1}^N \frac{1}{N_t^{i}} \Delta_x \frac{\delta \mathcal{V}}{\delta \mu}(t, \mu^N_t, X_t^{i}) + \frac{1}{2} \sum_{i=1}^N \frac{1}{(N_t^{i})^2} \tr \nabla^2_{xy} \frac{\delta^2 \mathcal{V}}{\delta \mu^2}(t,\mu^N_t,X_t^{i}, X_t^{i})  dt \\
    &+  \int_{t_0}^T \sum_{i=1}^N  \frac{1}{N_t^{i}} \nabla_x \frac{\delta \mathcal{V}}{\delta \mu}(t,\mu^N_t, X_t^{i}) \cdot dB_t^{i}.
\end{align*}
Using the equation satisfied by $\mathcal{V}$, we get
\begin{align*}
    \mathcal{V}(t_0, \mu_0^N) &= \sum_{i=1}^N \frac{1}{N_T^{i}} g(X_T^{i}) - \int_{t_0}^T  \sum_{i=1}^N \frac{1}{N_t^{i}} \bigl( \alpha_t^{i} \cdot \nabla_x \frac{\delta \mathcal{V}}{\delta \mu} (t, \mu^N_t , X_t^{i}) + \frac{1}{2}|\nabla_x \frac{\delta \mathcal{V}}{\delta \mu}(t,\mu_t^N,X_t^{i}) |^2 \bigr) dt \\
    &- \frac{1}{2}\int_{t_0}^T \sum_{i=1}^N \frac{1}{(N_t^{i})^2} \tr \nabla^2_{xy} \frac{\delta^2 \mathcal{V}}{\delta \mu^2}(t,\mu^N_t,X_t^{i}, X_t^{i})  dt \\
    &-  \int_{t_0}^T \sum_{i=1}^N  \frac{1}{N_t^{i}} \nabla_x \frac{\delta \mathcal{V}}{\delta \mu}(t,\mu^N_t, X_t^{i}) \cdot dB_t^{i},
\end{align*}
which can be rewritten into
\begin{align*}
    \mathcal{V}(t_0, \mu_0^N) &= \sum_{i=1}^N \frac{1}{N_T^{i}} g(X_T^{i}) + \int_{t_0}^T \sum_{i=1}^N \frac{1}{N_t^{i}} \frac{1}{2} |\alpha_t^{i}|^2 dt -  \frac{1}{2}\int_{t_0}^T  \sum_{i=1}^N \frac{1}{N_t^{i}} \bigl| \alpha_t^{i} + \nabla_x \frac{\delta \mathcal{V}}{\delta \mu} (t, \mu^N_t , X_t^{i}) \bigr|^2dt \\
    &-\frac{1}{2}\int_{t_0}^T \sum_{i=1}^N \frac{1}{(N_t^{i})^2} \tr \nabla^2_{xy} \frac{\delta^2 \mathcal{V}}{\delta \mu^2}(t,\mu^N_t,X_t^{i}, X_t^{i})  dt \\
    &-  \int_{t_0}^T \sum_{i=1}^N  \frac{1}{N_t^{i}} \nabla_x \frac{\delta \mathcal{V}}{\delta \mu}(t,\mu^N_t, X_t^{i}) \cdot dB_t^{i}.
\end{align*}
Taking expectations we obtain
\begin{align}
    \mathcal{V}(t_0, \mu_0^N) &= J^N(t_0, \bx,\ba,(\alpha_t)^N_{t \geq t_0}) - \E \bigl[ \frac{1}{2}\int_{t_0}^T  \sum_{i=1}^N \frac{1}{N_t^{i}} \bigl| \alpha_t^{i} + \nabla_x \frac{\delta \mathcal{V}}{\delta \mu} (t, \mu^N_t , X_t^{i}) \bigr|^2dt \bigr]  \label{eq:towardcomparisonvf} \\
 \notag   &- \E \bigl[ \frac{1}{2} \int_{t_0}^T \sum_{i=1}^N \frac{1}{(N_t^{i})^2} \tr \nabla^2_{xy} \frac{\delta^2 \mathcal{V}}{\delta \mu^2}(t,\mu^N_t,X_t^{i}, X_t^{i})  dt \bigr].
\end{align}
Now we make the following observation. If $V$ is bounded there is $C>0$ depending on $\norm{V}_{\mathcal{C}^0(\T^d)}$ and $T$ such that, whatever the controls $(\alpha_t^1, \dots, \alpha_t^N)_{t \geq t_0}$ it holds
\begin{equation}
     \frac{1}{N_t^{i}} \leq \frac{C}{n^{i}[\ba]}, \quad \forall i=1, \dots,N, \quad \forall t \in [t_0,T].
\label{eq:constantforai}
\end{equation}
[recall the $A_t^{i} = a^{i} + \int_{t_0}^t V(X_s^{i})ds$]. 
As a consequence, taking $ \alpha_t^{i} = - \nabla_x \frac{\delta \mathcal{V}}{\delta \mu}(t, \mu_t^N,X^{i}_t)$ in \eqref{eq:towardcomparisonvf}, we have 
$$v^N(t_0,\bx,\ba) \leq J^N(t_0, \bx,\ba,(\alpha_t)^N_{t \geq t_0}) \leq \mathcal{V}(t_0, \mu_0^N) + C \sum_{i=1}^N \bigl ( \frac{1}{n^{i}[\ba]} \bigr)^2,$$
for some $C>0$ depending on the constant appearing in \eqref{eq:constantforai} and on  $\norm{\nabla^2_{xy} \frac{\delta^2 \mathcal{V}}{\delta \mu^2}}_{L^{\infty}}$. If we take the infimum over all possible controls $(\alpha_t^1, \dots, \alpha_t^N)_{t \geq t_0}$ in \eqref{eq:towardcomparisonvf} we get
$$ \mathcal{V}(t_0, \mu_0^N) \leq v^N(t_0,\bx,\ba) + C\sum_{i=1}^N \bigl ( \frac{1}{n^{i}[\ba]} \bigr)^2.$$
Combining the two inequalities leads to
$$ \bigl| \mathcal{V}(t_0, \mu^N_{\bx,\ba}) - v^N(t_0, \bx,\ba) \bigr| \leq C \sum_{i=1}^N \bigl ( \frac{1}{n^{i}[\ba]} \bigr)^2.$$
Unfortunately, this argument is completely formal unless we were able to prove that $\mathcal{V}$ has enough regularity.

\section{The sup-convolution and its properties.}
\label{sec:supconvolandproperties16/04}

In order to make the formal argument of the previous section rigorous, we introduce a suitable mollification of $\mathcal{V}$, and apply the projection argument to this mollified version. The obvious difficulty is to find a mollification procedure that is \textit{compatible} with the equation. That is, we need the mollified function to satisfy a similar equation, at least up to some error terms that can be explicitly controlled. 

Similarly to \cite{ddj2023}, for $\epsilon >0$ we define 
\begin{equation} 
\mathcal{V}^{\epsilon}(t,q) := \sup_{ \mu \in \mathcal{P}(\T^d)} \{ \mathcal{V}(t,\mu) - \frac{1}{2\epsilon} \norm{q-\mu}^2_{H^{-k}(\T^d)} \}, \qquad  q \in H^{-k}(\T^d), \; t \in [t_0,T],
\label{eq:defnVeps18/04}
\end{equation}
where we recall that $H^{-k}(\T^d)$ and its metrics were introduced at the beginning of Section \ref{sec:AnalysisVF}. Notice that $\mathcal{V}^{\epsilon}$ is now defined over $[0,T] \times H^{-k}(\T^d)$. 

Before we go on with the analysis of $\mathcal{V}^{\epsilon}$, let us recall some facts about the analysis and calculus over the space $H^{-k}(\T^d)$. The material is taken from \cite[Section 2.5]{ddj2023}.

\subsection{Some preliminary facts about $H^{-k}(\T^d)$.}

We recall that $H^k(\T^d)$ is the space of functions with derivatives up to order $k$ in $L^2(\T^d)$, endowed with inner product
$$ \lg \phi, \psi \rg_k := \sum_{|j| \leq k} \int_{\T^d} \nabla^j \phi(x) \nabla^j\psi (x) dx,  $$
where the sum is taken over all the multi indices $ j = (j^1, \dots, j^d) \in \mathbb{N}^d$ with $|j| := j^1 + \dots + j^d \leq k$ and $\nabla^j := \partial^{j^1}_{x^1}\cdots \partial^{j^d}_{x^d} $. We denote by $\norm{ \cdot}_{H^k(\T^d)}$ the associated norm. 
The dual space of $H^k(\T^d)$, i.e. the space of bounded linear functionnals over $H^k(\T^d)$, is denoted by $H^{-k}(\T^d)$, and $\norm{ \cdot }_{H^{-k}(\T^d)}$ is the norm inherited by duality, that is
\begin{equation} 
\norm{q}_{H^{-k}(\T^d)} := \sup_{ \phi \in H^k(\T^d), \norm{\phi}_{H^{k}(\T^d)} \leq 1} q(\phi).
\label{eq:dualitynorm}
\end{equation}
If $q$ is an element of $H^{-k}(\T^d)$, we denote by $q^*$ the unique element of $H^k(\T^d)$ such that 
$$ q(\phi) = \lg \phi, q^* \rg_k \quad \mbox{ for all } \quad \phi \in H^k(\T^d).$$
Then the map $q \mapsto q^*$ is an isometry from $H^{-k}(\T^d)$ onto $H^k(\T^d)$ when $H^{-k}$ is endowed with the inner product $\lg q, p \rg_{-k} := \lg q^*, p^* \rg_k $ which is consistent with the norm \eqref{eq:dualitynorm} inherited by duality.

When $k > d/2 +2$, $H^k(\T^d)$ continuously embeds into $\mathcal{C}^2(\T^d)$. In particular, $\mathcal{P}(\T^d)$ can be seen as a (compact) subset of $H^{-k}(\T^d)$ via the identification $\mu(\phi) := \int_{\T^d} \phi(x) d\mu(x)$ for every $\phi \in H^k(\T^d)$.

We say that $\Phi : H^{-k}(\T^d) \rightarrow \R$ is differentiable at $q \in H^{-k}(\T^d)$ if it is Fr\'echet differentiable at this point, and we denote by $D_{-k} \Phi(q) \in H^k(\T^d)$ its derivative, by $\nabla_{-k} \Phi(q) = (D_{-k} \Phi(q))^* \in H^{-k}(\T^d)$ its gradient, and we denote by $\mathcal{C}^1(H^{-k}(\T^d))$ the space of continuously Fr\'echet differentiable functions over $H^{-k}(\T^d)$.

In Section \ref{sec:AnalysisVF}, before stating Lemma \ref{lem:subsolpropV16/04}, we introduced a notion of derivative for functions $\Phi$ defined over $\mathcal{P}(\T^d)$. The link between the two notions is as follows, see \cite[Lemma 2.15]{ddj2023}: if we assume that $\Phi$ belongs to $\mathcal{C}^1(H^{-k}(\T^d))$ for $k > d/2 +1$, the restriction $\Phi|_{\mathcal{P}(\T^d)}$ belongs to $\mathcal{C}^1(\mathcal{P}(\T^d))$ and 
\begin{equation} 
\frac{\delta \Phi}{\delta \mu}(\mu,x) = D_{-k}\Phi(\mu)(x) - \int_{\T^d} D_{-k}\Phi(\mu)(y) dy.  \label{eq:linkderivatives16/04}
\end{equation}
Notice that the second term in the right-hand side of \eqref{eq:linkderivatives16/04} comes from our normalization convention \eqref{eq:normalizingconvention16/04} for the linear derivative.  Moreover, by Sobolev embeddings,  $x \mapsto \frac{\delta \Phi}{\delta \mu}(\mu,x)$ is $\mathcal{C}^1$ with jointly continuous derivatives. We can easily extend \cite[Lemma 2.15]{ddj2023} when $k >d/2 +2$, in which case $x \mapsto \frac{\delta \Phi}{\delta \mu}(\mu,x)$ is $\mathcal{C}^2$ with jointly continuous derivatives. In particular, when $k > d/2 + 2$ (the restriction to $\mathcal{P}(\T^d)$ of) a function $\Phi \in \mathcal{C}^1(H^{-k}(\T^d))$ is a smooth test function in the sense defined before Lemma \ref{lem:subsolpropV16/04}.
If $\Phi$ belongs to $\mathcal{C}^1( H^{-k}(\T^d))$ we also denote 
\begin{equation}
    [\Phi]_{\mathcal{C}^{1,1}(H^{-k})} := \sup_{q_1 \neq q_2 \in H^{-k}} \frac{\norm{ \nabla_{-k} \Phi(q_2) - \nabla_{-k}\Phi(q_1)}_{H^{-k}(\T^d)}}{\norm{ q_2-q_1}_{H^{-k}(\T^d)}}
\label{eq:notationC11_17/04}
\end{equation}
the Lipschitz norm of its gradient.

The next lemma, stated and proved in \cite[Proposition 2.16]{ddj2023}  justifies the use of the metric $\norm{ \cdot}_{H^{-k}(\T^d)}$ in the regularization procedure.

\begin{lem}
\label{lem:C11controlsD2mm-19/04}
    Assume that $k > d/2 +1$. Then there is a constant $C_{k,d}>0$ depending only on $k$ and $d$ such that, for all $\Phi \in \mathcal{C}^1( H^{-k}(\T^d))$ and all $\mu^1,\mu^2 \in \mathcal{P}(\T^d)$ we have
    $$ \sup_{x \in \T^d} | \nabla_x \frac{\delta \Phi}{\delta\mu}(\mu^2,x) - \nabla_x \frac{\delta \Phi}{ \delta \mu}(\mu^1,x) | \leq C_{k,d} [\Phi]_{\mathcal{C}^{1,1}(H^{-k})} d_1(\mu^1,\mu^2).  $$
\end{lem}

To see the link with the projection argument of the previous section, recall \eqref{eq:towardcomparisonvf} and observe that, in the previous Lemma, if $\Phi$ is twice differentiable with respect to the measure argument over $\mathcal{P}(\T^d)$ this translates into the bound
$$ \sup_{x,y} | \nabla^2_{xy} \frac{\delta^2 \Phi}{\delta \mu^2}(\mu,x,y) | \leq C_{k,d} \;[\Phi]_{\mathcal{C}^{1,1}(H^{-k})} 
$$
for the same constant $C_{k,d}$. In other words, if we can produce $\mathcal{C}^{1,1}$-regularity in $H^{-k}(\T^d)$, we can automatically control the term involving the second order linear derivative of the value function that determines the rate of convergence.

\subsection{Regularity Properties of $\mathcal{V}^{\epsilon}$.}

Applying \cite[Proposition 4.3]{ddj2023}, we find that $\mathcal{V}^{\epsilon}(t, \cdot)$ belongs to $\mathcal{C}^{1,1}(H^{-k}(\T^d))$ and we can quantify the blow up of the $\mathcal{C}^{1,1}$ norm with respect to $\epsilon$. In the next statement, which is identical to \cite[Proposition 4.3]{ddj2023}, it is convenient to denote respectively by $C_L$ and $C_S$ the Lipschitz and semi-concavity constants of $\mathcal{V}$ with respect to $\norm{ \cdot }_{H^{k}(\T^d)}$. 

\begin{prop}\label{prop:supconvproperties} 
For all $\epsilon < \frac{1}{2 C_S}$ and all $t \in [0,T]$ we have 
\begin{enumerate}
    \item For all $\mu \in \mathcal{P}(\T^d)$,
$$ 0 \leq \mathcal{V}^{\epsilon}(t,\mu) - \mathcal{V}(t,\mu) \leq 2C_L^2 \epsilon.$$

    \item $\mathcal{V}^{\epsilon}(t, \cdot)$ belongs to  $\mathcal{C}^{1}(H^{-k}(\T^d))$, and, with the notation \eqref{eq:notationC11_17/04}, 
    \begin{align} \label{c11bound}
    \big[\mathcal{V}^{\epsilon}(t,\cdot)\big]_{\mathcal{C}^{1,1}(H^{-k})} \leq  \frac{1}{\epsilon}.
    \end{align}
    \item For $\mu \in \mathcal{P}(\T^d)$, the gradient of $\mathcal{V}^{\epsilon}(t,\cdot)$ at $\mu$ is given by
    \begin{align*}
        \nabla_{-k} \mathcal{V}^{\epsilon}(t,\mu) = \frac{1}{\epsilon} \big(\mu_{\epsilon} - \mu\big), 
    \end{align*}
    where $\mu_{\epsilon}$ is the unique element of $\mathcal{P}(\T^d)$ such that 
    \begin{align*}
        \mathcal{V}^{\epsilon}(t,\mu) = \mathcal{V}(t,\mu_{\epsilon}) - \frac{1}{2 \epsilon} \norm{\mu_{\epsilon} - \mu}_{-k}^2.
    \end{align*}
    and $\mu_{\epsilon}$ satisfies $\norm{\mu - \mu_{\epsilon}}_{-k} \leq 2 C_L \epsilon$.

     \item $\mathcal{V}^{\epsilon}(t,\cdot)$ is Lipschitz continuous over $\mathcal{P}(\T^d)$. More precisely, for all $\mu_1, \mu_2 \in \mathcal{P}(\T^d)$,
$$ |\mathcal{V}^{\epsilon}(t,\mu_1)- \mathcal{V}^{\epsilon}(t,\mu_2) | \leq 2C_L \|\mu_1- \mu_2 \|_{-k}.$$

\end{enumerate}
\end{prop}

We proceed to show the sub-solution property for $\mathcal{V}^{\epsilon}$.

\begin{prop}
 $\mathcal{V}^{\epsilon}$ is a sub-solution over $[0,T] \times \mathcal{P}(\T^d)$ to the dynamic programming equation  \eqref{eq:MasterHJB16/04} for $\mathcal{V}$ up to an error of order $ C\epsilon$. More precisely, there is a constant $C>0$ independent of $\epsilon >0$ such that, for all $\Phi \in \mathcal{C}^1([0,T] \times H^{-k}(\T^d))$ such that $\mathcal{V}^{\epsilon} - \Phi$ has a global maximum at $(t_0,\mu_0) \in [0,T] \times \mathcal{P}(\T^d)$ it holds
    \begin{align*} 
- \partial_t \Phi(t_0,\mu_0) -\frac{1}{2} \int_{\T^d} \Delta_x \frac{\delta \Phi}{\delta \mu}(t_0,\mu_0,x) &d\mu_0(x) + \frac{1}{2} \int_{\T^d} \bigl|\nabla_x \frac{\delta \Phi}{\delta \mu}(t,\mu_0,x) \bigr|^2 d\mu_0(x) \\
&+ \int_{\T^d} \bigl( V(x) - \lg V,\mu_0 \rg \bigr) \frac{\delta \mathcal{V}}{\delta \mu}(t_0, \mu_0,x) d\mu_0(x) \leq C \epsilon. 
\end{align*}
\end{prop}

\begin{proof}
Take $(t_0,\mu_0) \in [0,T] \times \mathcal{P}(\T^d)$ and assume that $\Phi : [0,T] \times H^{-k} (\T^d) \rightarrow \R$ is a $\mathcal{C}^1(H^{-k}(\T^d))$ function touching $\mathcal{V}^{\epsilon}$ from above at $(t_0,\mu_0)$, that is
    $$ \mathcal{V}^{\epsilon}(t_0,\mu_0) - \Phi(t_0,\mu_0) = \sup_{ (t,q) \in [0,T] \times H^{-k}(\T^d)} \mathcal{V}^{\epsilon}(t,q) - \Phi(t,q). $$
Now, we let $\mu^{\epsilon} \in \mathcal{P}(\T^d)$ be a point such that
$$\mathcal{V}^{\epsilon}(t_0, \mu_0) = \mathcal{V}(t_0, \mu^{\epsilon}) - \frac{1}{2\epsilon} \norm{\mu_0 - \mu^{\epsilon} }_{H^{-k}(\T^d)}^2,$$
and we easily check that $(t_0,\mu_0, \mu^{\epsilon})$ is a maximum, over $[0,T] \times H^{-k}(\T^d) \times \mathcal{P}(\T^d)$ of 
\begin{equation} 
(t,q,\mu) \mapsto \mathcal{V}(t,\mu) - \frac{1}{2\epsilon} \norm{q - \mu}_{-k}^2 - \Phi(t,q).
\label{eq:D-kPhi16/04}
\end{equation}
In particular, $ \mu_0$ is a maximum over $H^{-k}(\T^d)$ of 
$$
q \mapsto  - \frac{1}{2\epsilon} \norm{q-\mu^{\epsilon} }^2_{-k} - \Phi(t_0,q),
$$
and therefore, by regularity of $\Phi$,
$$ 
D_{-k} \Phi(t_0,\mu_0) = \frac{1}{\epsilon}( \mu^{\epsilon} - \mu_0)^{*} \in H^k(\T^d).
$$
Using the Lipschitz regularity of $\mathcal{V}$ from Proposition \ref{prop:lipandSCV16/004} and Proposition \ref{prop:supconvproperties} above we deduce that 
\begin{equation} 
\norm{ D_{-k} \Phi(t_0,\mu_0)}_{H^k(\T^d)} = \frac{1}{\epsilon} \norm{ \mu^{\epsilon} - \mu_0}_{H^{-k}(\T^d)} \leq 2C_L.   
\label{eq:16/04-10:14}
\end{equation}
We let $\Phi^{\epsilon}(t,\mu) := \Phi(t, \mu + \mu_0 - \mu^{\epsilon}) $ and we check that $(t_0, \mu^{\epsilon})$ is a maximum over $[0,T] \times \mathcal{P}(\T^d)$ of
$$
(t,\mu) \mapsto \mathcal{V}(t,\mu) - \Phi^{\epsilon}(t,\mu).
$$
$\Phi^{\epsilon}$ is a smooth test function in the sense of Lemma \ref{lem:subsolpropV16/04}, and we can use the sub-solution property of $\mathcal{V}$ at $(t_0,\mu^{\epsilon})$ to infer that
\begin{align*}  
-\partial_t \Phi^{\epsilon}(t_0,\mu^{\epsilon}) - \frac{1}{2}\int_{\T^d} \Delta_x \frac{\delta \Phi^{\epsilon}}{\delta \mu}(t_0,\mu^{\epsilon},x)d\mu^{\epsilon}(x) +  \frac{1}{2} \int_{\T^d} \bigl|\nabla_x \frac{\delta \Phi^{\epsilon}}{\delta \mu}(t_0,\mu^{\epsilon},x) |^2 d\mu^{\epsilon}(x) \\
\hspace{140pt} + \int_{\T^d} \bigl( V(x) - \lg V;\mu^{\epsilon} \rg \bigr) \frac{\delta \Phi^{\epsilon}}{\delta \mu}(t,\mu^{\epsilon},x) d\mu^{\epsilon}(x) \leq 0.
\end{align*}
By definition of $\Phi^{\epsilon}$ we easily have $\frac{\delta \Phi^{\epsilon}}{\delta \mu}(t_0,\mu^{\epsilon})= \frac{\delta \Phi}{\delta \mu}(t_0, \mu_0)$ and $\partial_t \Phi^{\epsilon}(t_0,\mu^{\epsilon}) = \partial_t \Phi(t_0,\mu_0)$ and therefore
\begin{align*}  
-\partial_t \Phi(t_0,\mu_0) - \frac{1}{2}\int_{\T^d} \Delta_x \frac{\delta \Phi}{\delta \mu}(t_0,\mu_0,x)d\mu_0(x) + \frac{1}{2} \int_{\T^d}  \bigl|\nabla_x \frac{\delta \Phi}{\delta \mu}(t_0,\mu_0,x) |^2 d\mu_0(x) \\
\hspace{140pt} + \int_{\T^d} \bigl( V(x) - \lg V;\mu_0 \rg \bigr) \frac{\delta \Phi}{\delta \mu}(t,\mu_0,x) d\mu_0(x) \leq E_1 + E_2 + E_3 + E_4,
\end{align*}
with $E_1$, $E_2$, $E_3$ and $E_4$ defined by
\begin{align*}
E_1 &:= \frac{1}{2} \int_{\T^d} \Delta_x \frac{\delta \Phi}{\delta \mu}(t_0,\mu_0,x)d(\mu^{\epsilon} - \mu_0)(x),\\
 E_2 &:= \frac{1}{2} \int_{\T^d} \bigl| \nabla_x \frac{\delta \Phi}{\delta \mu}(t,\mu_0,x) \bigr|^2 d(\mu_0 - \mu^{\epsilon})(x), \\
E_3 &:= \int_{\T^d} \lg V; \mu^{\epsilon} - \mu_0 \rg \frac{\delta \Phi}{\delta \mu}(t_0,\mu_0,x) d\mu_0(x),\\
E_4 &:= \int_{\T^d} \bigl( V(x) - \lg V; \mu^{\epsilon} \rg \bigr) \frac{\delta \Phi}{\delta \mu}(t_0,\mu_0,x)d(\mu^{\epsilon}- \mu_0)(x).
\end{align*}
For $E_1$ and $E_2$, we proceed as in \cite[Proposition 5.1 ]{ddj2023} (see also  Lemmas 5.4 and 5.5 therein). On the one hand, the Laplacian being  negative over $H^{k}(\T^d)$ we have
$$ E_1 =  \epsilon \frac{1}{2} \lg \Delta_x D_{-k} \Phi(t_0,\mu_0), D_{-k}\Phi(t_0,\mu_0) \rangle_k \leq 0.   $$
On the other hand, $E_2$ can be rewritten as
$$ E_2 = \frac{\epsilon}{2} \lg |\nabla_x (D_{-k} \Phi(t_0,\mu_0))|^2, D_{-k} \Phi(t_0,\mu_0) \rangle_k  $$
and, by \cite[Lemma 5.5]{ddj2023} we can find $C>0$ depending only on  $\norm{  D_{-k} \Phi(t_0,\mu_0) }_{H^{k}(\T^d)}$ (and therefore independent from $\epsilon$ thanks to \eqref{eq:16/04-10:14}) such that 
$$ 
|\lg |\nabla_x (D_{-k} \Phi(t_0,\mu_0))|^2, D_{-k} \Phi(t_0,\mu_0) \rangle_k| \leq C,
$$
and therefore, $E_2 \leq C \epsilon$ for some $C>0$ depending only on $C_L$, the Lipschitz constant of $\mathcal{V}$ with respect to $\norm{ \cdot}_{H^{-k}(\T^d)}$.
For $E_3$ we have
\begin{align*}
    E_3 &\leq \norm{V}_{H^k(\T^d)} \norm{ \mu^{\epsilon} - \mu_0}_{H^{-k}(\T^d)} \norm{ \frac{\delta \Phi}{\delta \mu}(t_0,\mu_0)}_{\mathcal{C}^0(\T^d)}.
\end{align*}
Since $k > d/2 +2$, Sobolev embedding gives:
\begin{align*} 
\norm{  \frac{\delta \Phi}{\delta \mu}(t_0,\mu_0)}_{\mathcal{C}^0(\T^d)} & = \norm{  D_{-k} \Phi (t_0,\mu_0) - \int_{\T^d} D_{-k} \Phi(t_0,\mu_0)(y)dy   }_{\mathcal{C}^0(\T^d)} \\
&\leq 2 \norm{  D_{-k} \Phi (t_0,\mu_0)  }_{\mathcal{C}^0(\T^d)} \\
&\leq C_{k,d} \norm{  D_{-k} \Phi (t_0,\mu_0)  }_{H^k(\T^d)} 
\end{align*}
for some constant $C_{k,d} >0$ depending only on $k$ and $d$. Combined with \eqref{eq:16/04-10:14} we find $C>0$ independent from $\epsilon >0$ such that $E_3 \leq C \epsilon.$
We proceed similarly for for $E_4$,
\begin{align*}
E_4 &\leq 2 \norm{ V}_{\mathcal{C}^0(\T^d)} \norm{ \frac{\delta \Phi}{\delta \mu}(t_0,\mu_0,\cdot)}_{H^k(\T^d)} \norm{\mu^{\epsilon}- \mu_0}_{H^{-k}(\T^d)} \\
&\leq 4 \epsilon \norm{V}_{\mathcal{C}^0(\T^d)} \norm{ D_{-k} \Phi(t_0,\mu_0)}_{H^k(\T^d)} \frac{1}{\epsilon} \norm{ \mu^{\epsilon} - \mu_0}_{H^{-k}(\T^d)}
\end{align*}
and we conclude from \eqref{eq:16/04-10:14} that $E_4 \leq C \epsilon$ for some $C>0$ independent from $\epsilon >0$.
Combining these estimates we deduce that $\mathcal{V}^{\epsilon}$ is a sub-solution over $[0,T] \times \mathcal{P}(\T^d)$ up to an error of size $C \epsilon$.
\end{proof}

\section{The hard inequality.}

\label{sec:hardineq}

The goal of this section is to prove the following upper bound on the difference $\mathcal{V} - v^N$.
\begin{prop} 
\label{prop:hardineq19/04}
   Assume that $V$ and $g$ belong to $\mathcal{C}^k(\T^d)$ for some $k > d/2 +2$. There is a constant $C$ depending on $T$, $\norm{V}_{\mathcal{C}^k(\T^d)}$ and $\norm{g}_{\mathcal{C}^k(\T^d)}$ such that, for every $(t, \bx, \ba) \in [0,T] \times (\T^d)^N \times \R^N$,
$$ \mathcal{V}(t, \mu^N_{\bx,\ba}) \leq v^N(t,\bx,\ba) + C \sqrt{ \sum_{i=1}^N \bigl( \frac{1}{n^{i}[\ba]} \bigr)^2    }.  $$ 
\end{prop}
Recall that the probability measure $\mu_{\bx, \ba}^N$ and the variables $n^{i}[\ba]$ are defined in\eqref{fo:empirical} and \eqref{fo:weights}.
For all $(t,\bx,\ba) \in [0,T] \times (\T^d)^N \times \R^N$ we define the projection
\begin{equation} 
v^{\epsilon,N}(t,\bx,\ba) := \mathcal{V}^{\epsilon}(t, \mu^N_{\bx,\ba}).
\label{eq:defnvepsN18/04}
\end{equation}
Our goal is to show that $v^{\epsilon,N}$ is almost a classical sub-solution for the equation for $v^N$, up to an error that can be quantified in terms of $\epsilon$, $N$ and some bound on $\ba$, and then use a comparison argument.
We start by showing that $v^{\epsilon,N}$ inherits some regularity from the $\mathcal{C}^{1,1}$ regularity of $\mathcal{V}^{\epsilon}$ with respect to $\norm{ \cdot }_{H^{-k}(\T^d)}$.

\begin{lem}
\label{lem:regvepsN19/04}
    Let $v^{\epsilon,N} :[0,T] \times (\T^d)^N \times (\R)^N$ be defined by \eqref{eq:defnvepsN18/04}.Then
\begin{enumerate}
    \item $v^{\epsilon,N}$ is uniformly Lipschitz over $[0,T] \times (\T^d)^N \times (\R)^N$.
    \item $v^{\epsilon,N}$ is differentiable in $(\bx,\ba) \in (\T^d)^N \times (\R)^N$ with 
\begin{equation} 
\nabla_{x^{i}} v^{\epsilon,N}(t,\bx,\ba) = \frac{1}{n^{i}[\ba]} \nabla_x \bigl( D_{-k} \mathcal{V}^{\epsilon}(t,\mu^N_{\bx,\ba}) \bigr) (x^{i}) \label{eq:derivativeinxvepsN19/04}
\end{equation}
\begin{equation} 
\partial_{a^{i}} v^{\epsilon,N}(t, \bx,\ba) =-\frac{1}{n^{i}[\ba]} \Bigl(  D_{-k} \mathcal{V}^{\epsilon}(t,\mu^N_{\bx,\ba}) (x^{i})  - \int_{\T^d} D_{-k} \mathcal{V}^{\epsilon}(t,\mu^N_{\bx,\ba}) (x) d\mu^N_{\bx,\ba}(x) \Bigr).
\label{eq:derivativeinavepsN19/04}
\end{equation}
    \item The derivative $\nabla_{x^{i}}v^{\epsilon,N}$ is Lipschitz continuous in $\bx$ uniformly in $(t,\ba).$
    \item The derivatives $\partial_t v^{\epsilon,N}$, $\nabla_{x^{i}} v^{\epsilon,N}$, $\partial_{a^{i}} v^{\epsilon,N} $ and $\nabla^2_{x^{i}x^j} v^{\epsilon,N}$ exist almost everywhere and define versions of the weak derivatives of $v^{\epsilon,N}$. Moreover, these derivatives are uniformly bounded over $[0,T] \times (\T^d)^N \times (\R)^N$ for all $1 \leq i,j \leq N$.
\end{enumerate}
\end{lem}

\begin{proof}
    Point $\textit{(1)}$ is a consequence of the regularity of $\mathcal{V}^{\epsilon}$ given by Propositions \ref{prop:lipandSCV16/004} and \ref{prop:supconvproperties}. Notice that $\mathcal{V}$ is uniformly Lipschitz in time by Proposition \ref{prop:timeregV}. This is inherited by $\mathcal{V}^{\epsilon}$ and therefore by $v^{\epsilon,N}$. Combining Proposition \ref{prop:lipandSCV16/004} and Proposition \ref{prop:supconvproperties} we also know that $\mu \mapsto \mathcal{V}^{\epsilon}(t,\mu)$ is $\mathcal{C}^1$ in the $H^{-k}$ sense for all $t \in [0,T]$. By the chain rule we deduce that $(\bx,\ba) \mapsto v^{\epsilon,N}(t,\bx,\ba)$ is $\mathcal{C}^1$ over $(\T^d)^N \times \R^N$ for all $t \in [0,T]$ and the derivatives are given by \eqref{eq:derivativeinxvepsN19/04} and \eqref{eq:derivativeinavepsN19/04}. This proves Point $\textit{(2)}$ in the Lemma. Point $\textit{(3)}$ is a direct consequence of Lemma \ref{lem:C11controlsD2mm-19/04} and the $\mathcal{C}^{1,1}$ regularity of $\mathcal{V}^{\epsilon}$ given by Proposition \ref{prop:supconvproperties}.  Point $\textit{(4)}$ is then a consequence of points $\textit{(1)}$ and $\textit{(3)}$ together with Rademacher's theorem.
\end{proof}
Next we need a technical Lemma that allows us to bypass the lack of second order differentiability of $\mathcal{V}^{\epsilon}$ with respect to the measure variable. The proof is identical to the one of \cite[Lemma 5.14]{ddj2023} and so it is omitted.

\begin{lem}
    For almost all $(t,\bx,\ba) \in [0,T] \times (\T^d)^N \times \R^N$, for all $1 \leq i \leq N$, the function
\begin{equation} 
\T^d \ni x \mapsto \nabla_x \bigl( D_{-k} \mathcal{V}^{\epsilon}(t, \sum_{j=1, j \neq i}^N \frac{1}{n^{j}[\ba]} \delta_{x^{j}} + \frac{1}{n^{i}[\ba]} \delta_x) \bigr) (x^{i})  = \nabla_x  \frac{\delta \mathcal{V}^{\epsilon}}{\delta \mu} (t, \sum_{j=1, j \neq i}^N \frac{1}{n^{j}[\ba]} \delta_{x^{j}} + \frac{1}{n^{i}[\ba]} \delta_x, x^{i}) 
\label{eq:thefunction19/04}
\end{equation}
is differentiable at $x^{i}$. At any point $(t,\bx, \ba)$ such that \eqref{eq:thefunction19/04} is differentiable at $x^{i}$, the second derivative $\nabla^2_{x^{i}x^{i}} v^{\epsilon,N}$ exists (in the classical sense) and is given by 
$$\nabla^2_{x^{i}x^{i}} v^{\epsilon,N} (t, \bx, \ba)= \frac{1}{n^{i}[\ba]} \nabla^2_{x} \frac{\delta \mathcal{V}^{\epsilon}}{\delta \mu}(t,\mu^N_{\bx,\ba},x^{i}) + \bigl( \frac{1}{n^{i}[\ba]} \bigr)^2 R^{N,i}(t,\bx,\ba)  $$
where $R^{N,i}(t,\bx,\ba)$ is explicitly given by  
$$ R^{N,i}(t,\bx,\ba) = n^{i}[\ba] \nabla_x \bigl[ \T^d \ni x \mapsto  \nabla_x  \frac{\delta \mathcal{V}^{\epsilon}}{\delta \mu} (t, \sum_{j=1, j \neq i}^N \frac{1}{n^{j}[\ba]} \delta_{x^{j}} + \frac{1}{n^{i}[\ba]} \delta_x)(x^{i}) \bigr)   \bigr] $$
and satisfies $\norm{ R^{N,i}}_{L^{\infty}} \leq C \sup_{t \in [0,T]} [ \mathcal{V}^{\epsilon}]_{\mathcal{C}^{1,1}(H^{-k})}. $
In particular, by Proposition \ref{prop:supconvproperties}, $\norm{ R^{N,i}}_{L^{\infty}} \leq C/\epsilon$ for some $C>0$ independent from $N$.
\end{lem}

Following \cite[Proposition 5.10]{ddj2023}  we then deduce that $v^{\epsilon,N}$ is a sub-solution to the equation satisfied by $v^N$ up to an error term of order $\frac{C}{\epsilon}  \sum_{i=1}^N \bigl(\frac{1}{n^{i}[\ba]} \bigr)^2.$

\begin{lem}
\label{lem:subsolvepsN19/04}
There is some $C>0$ independent from $N$ and $\epsilon$ such that, for almost every $(t,\bx,\ba) \in [0,T] \times (\T^d)^N \times \R^N$ it holds
\begin{align*}
 \displaystyle    -\partial_t v^{\epsilon,N}(t,\bx,\ba) -  \frac{1}{2}\sum_{i=1}^N &\Delta_{x^{i}} v^{\epsilon,N}(t,\bx,\ba) - \sum_{i=1}^N V(x^{i}) \partial_{a^{i}} v^{\epsilon,N}(t,\bx,\ba) \\
 &+  \sum_{i=1}^N  \frac{n^{i}[\ba]}{2} |\nabla_{x^{i}} v^{\epsilon,N}(t,\bx,\ba)|^2  \leq \frac{C}{\epsilon}  \sum_{i=1}^N \bigl(\frac{1}{n^{i}[\ba]} \bigr)^2.  
\end{align*}
\end{lem}

We can finally prove Proposition \ref{prop:hardineq19/04}. 

\begin{proof}[Proof of Proposition \ref{prop:hardineq19/04}.]
    Take $(t_0,\bx_0,\ba_0) \in [0,T] \times (\T^d)^N \times (\R)^N$ and let $(\alpha^1_t,\dots,\alpha^N_t)_{t \in [t_0,T]}$ be an admissible control for $ (t_0,\bx_0, \ba_0)$ in the $N$-particle problem defining $v^N$, and let $(X^1_t,\dots,X_t^N)_{ t \in [t_0,T]}$ be the be the processes defined by:
$$ 
dX_t^{i} = \alpha_t^{i} dt + dB_t^{i}, \quad t \geq t_0, \quad X_{t_0}^{i} = x_0^{i}, \quad 1 \leq i \leq N.
$$
We then let $(A_t^{1}, \dots, A_t^N)_{t \in [t_0,T]}$ be the solutions to 
$$ dA_t^{i} = V(X_t^{i}) dt, \quad t \geq t_0, \quad A_{t_0}^{i} = a_0^{i}, \quad 1 \leq i \leq N$$
and introduce the variables 
$$ N_t^{i} = e^{A_t^{i}} \sum_{j=1}^N e^{-A_t^{j}}, \quad t \geq t_0, \quad 1 \leq i \leq N. $$
The function $v^{\epsilon,N}$ has weak derivatives $\partial_t v^{\epsilon,N}$, $\partial_{a^{i}} v^{\epsilon,N}$, $\nabla_{x^{i}} v^{\epsilon,N}, \nabla^2_{x^{i}x^{j}} v^{\epsilon,N} $ in $L^{\infty}([t_0,T] \times (\T^d)^N \times (\R)^N)$ by Lemma \ref{lem:regvepsN19/04}. Then we can apply Itô-Krylov formula, \cite[Section 2.10]{Krylov1980}, to infer that for all $t \in [t_0,T]$,
\begin{align*}
     v^{\epsilon,N}(t_0, \bx_0,\ba_0) &= \E v^{\epsilon,N}(T, \bX_T,\bA_T) \\
     &- \E  \int_{t_0}^T \Bigl \{ \partial_t v^{\epsilon,N}_t + \sum_{i=1}^N V(X_t^{i}) \partial_{a^{i}} v_t^{\epsilon,N} + \sum_{i=1}^N \alpha_t^{i} \cdot \nabla_{x^{i}} v_t^{\epsilon,N} + \frac{1}{2} \sum_{i=1}^N \Delta_{x^{i }} v_t^{\epsilon,N} \Bigr\}       (\bX_t,\bA_t) dt 
\end{align*}
Using Lemma \ref{lem:subsolvepsN19/04} we get
\begin{align*}
 v^{\epsilon,N}(t_0, \bx_0,\ba_0)     &\leq \E v^{\epsilon,N}(T,\bX_T, \bA_T) - \E \int_{t_0}^T \Bigl \{ \sum_{i=1}^N \frac{N_t^{i}}{2} |\nabla_{x^{i}} v_t^{\epsilon,N} |^2 + \alpha_t^{i} \cdot \nabla_{x^{i}} v_t^{\epsilon,N} \Bigr \}(\bX_t,\bA_t) dt \\
     &+  \frac{C}{\epsilon} \E \int_{t_0}^T \sum_{i=1}^N \bigl( \frac{1}{N_t^{i}} \bigr)^2 dt. 
\end{align*}
For all $t \in [t_0,T]$ and all $1 \leq i \leq N$, almost-surely,
$$ - \Bigl \{ \sum_{i=1}^N \frac{N_t^{i}}{2} |\nabla_{x^{i}} v_t^{\epsilon,N} |^2 + \alpha_t^{i} \cdot \nabla_{x^{i}} v_t^{\epsilon,N} \Bigr \}(\bX_t,\bA_t) \leq \frac{1}{2 N_t^{i}} |\alpha_t^{i}|^2  $$
and therefore, recalling the definition $J^N$ for the expected cost, we get
\begin{align} 
\notag v^{\epsilon,N}(t_0,\bx_0, \ba_0) &\leq J^N \bigl( (t_0,\bx_0,\ba_0), (\alpha^1_t\dots, \alpha^N_t)_{t \in [t_0,T]} ) \bigr) \\
&+ \E v^{\epsilon,N}_T(\bX_T, \bA_T) - \E \sum_{i=1}^N \frac{1}{N_T^{i}} g(X_t^{i})  +  \frac{C}{\epsilon} \E \int_{t_0}^T \sum_{i=1}^N \bigl( \frac{1}{N_t^{i}} \bigr)^2 dt.
\label{eq:19/04:18:32}
\end{align}
On the one hand, by Point $\textit{(1)}$ in Proposition \ref{prop:supconvproperties} we get 
$$ \sup_{(t, \bx, \ba) \in [0,T] \times (\T^d)^N \times (\R)^N} \bigl|  v^{\epsilon,N}(t,\bx,\ba) - \mathcal{V}(t, \mu^N_{\bx,\ba} )\bigr| \leq C \epsilon.  $$
On the other hand, since $V$ is bounded, there is $C>0$ depending on $\norm{V}_{\mathcal{C}^0(\T^d)}$ and $T >0$ (but not on $(\bx,\ba,(\alpha^1, \dots, \alpha^N))$ such that, almost-surely, for all $t \in [t_0,T]$,
$$ \frac{1}{N_t^{i}} = \frac{e^{-a_0^{i} -\int_{t_0}^T V(X_t^{i})dt}}{ \sum_{j=1}^N e^{-a_0^{j} -\int_{t_0}^T V(X_t^{j})dt} } \leq C \frac{e^{-a_0^{i}}}{ \sum_{j=1}^N e^{-a_0^{j}} } = \frac{C}{n^{i}[\ba_0]}. $$
Plugging this into \eqref{eq:19/04:18:32} we get
$$ \mathcal{V}(t_0,\mu^N_{\bx_0,\ba_0}) \leq J^N \bigl( (t_0,\bx_0,\ba_0), (\alpha^1_t\dots, \alpha^N_t)_{t \in [t_0,T]} ) \bigr) + C \epsilon + \frac{C}{\epsilon} \sum_{i=1}^N \bigl( \frac{1}{n^{i}[\ba_0]} \bigr)^2. $$
Taking the infimum over the admissible controls $(\alpha_t^1, \dots, \alpha_t^N)_{t \in [t_0,T]}$ we deduce that there is  some constant $C>0$ depending on $\norm{V}_{\mathcal{C}^k(\T^d)},\norm{ g}_{\mathcal{C}^k(\T^d)}$ and $T$, such that, for all $\epsilon >0$, for all $N \geq 1$ and for all $(t_0,\bx_0,\ba_0) \in [0,T] \times (\T^d)^N \times (\R)^N$,
$$ \mathcal{V}(t_0,\mu^N_{\bx_0,\ba_0}) \leq v^N(t_0, \bx_0,\ba_0) + C \epsilon + \frac{C}{\epsilon} \sum_{i=1}^N \bigl( \frac{1}{n^{i}[\ba_0]} \bigr)^2. $$
Now, for $N \geq 1$ and $(t_0,\bx_0,\ba_0) \in [0,T] \times (\T^d)^N \times (\R)^N$, we take 
$$\epsilon = \Bigl( \sum_{i=1}^N \bigl( \frac{1}{n^{i}[\ba_0]} \bigr)^2 \Bigr)^{1/2} \wedge \frac{1}{2 C_S} $$
to obtain the desired result.
\end{proof}

\section{The “easy" inequality.}

\label{sec:easyineq}

The goal of this section is to prove the following inequality.

\begin{prop} 
\label{prop:easyineq19/04}
   Assume that $V$ and $g$ belong to $\mathcal{C}^3(\T^d)$. There is a constant $C$ depending on $T$, $\norm{V}_{\mathcal{C}^3(\T^d)}$ and $\norm{g}_{\mathcal{C}^3(\T^d)}$ such that, for every $(t, \bx, \ba) \in [0,T] \times (\T^d)^N \times \R^N$,
$$ v^N(t,\bx,\ba) \leq \mathcal{V}(t, \mu^N_{\bx,\ba}) + C  \sum_{i=1}^N \bigl( \frac{1}{n^{i}[\ba]} \bigr)^2.  $$ 
\end{prop}

\begin{proof}
    We fix $(t_0,\bx_0,\ba_0) \in [0,T] \times (\T^d)^N \times \R^N$ and we let $\alpha : [t_0,T] \times \T^d \rightarrow \R^d$ be an optimal control for $\mathcal{V}(t_0,\mu^N_{\bx_0,\ba_0})$. For every $(t,\mu) \in [t_0,T] \times \mathcal{P}(\T^d)$ we let 
    $$ \Phi(t,\mu) := \int_t^T \int_{\T^d} \frac{1}{2} |\alpha_s(x)|^2 d\mu^{t,\mu}_s(x)ds + \int_{\T^d} g(x) d\mu_T^{t,\mu}(x)  $$
where $(\mu_s^{t,\mu})_{ s \in [t,T]}$ is the solution to 
$$ \partial_s \mu_s - \frac{1}{2} \Delta \mu_s + \div(\alpha_s \mu_s) + \bigl( V(x) - \langle V;\mu_s \rangle \bigr) \mu_s = 0, \mbox{ in } (t,T) \times \T^d, \quad \mu_t = \mu. $$ 
Arguing as in \cite{BuckdahnLiPengRainer} Theorem 7.2, see also  \cite{CST} Theorem 2.15, we can prove that  $\Phi : [t_0,T] \times \mathcal{P}(\T^d) \rightarrow \R$ is a classical solution to the (linear) equation 
\begin{equation*}
    \left \{
    \begin{array}{ll}
\displaystyle    -\partial_t \Phi(t,\mu)- \frac{1}{2}\int_{\T^d} \Delta_x \frac{\delta \Phi}{\delta \mu}(t,\mu,x) d\mu(x) - \int_{\T^d} \alpha_t(x) \cdot \nabla_x \frac{\delta \Phi}{\delta \mu}(t,\mu,x) d\mu(x) \\
 \displaystyle  \hspace{50pt} + \int_{\T^d}  (V(x) - \lg V;\mu \rg) \frac{\delta \Phi}{\delta \mu}(t,\mu,x) d\mu(x) = \frac{1}{2} \int_{\T^d} |\alpha_t(x)|^2 d\mu(x) , \quad \mbox{ in } [t_0,T] \times \mathcal{P}(\T^d), \\
\Phi(T, \mu) = \int_{\T^d} g(x) d\mu(x), \mbox{ in } \mathcal{P}(\T^d), 
    \end{array}
    \right.
\end{equation*}
and we can find $C >0$ depending only on $T$, $\norm{g}_{\mathcal{C}^3(\T^d)}, \norm{V}_{\mathcal{C}^3(\T^d)}$, $\sup_{t \in [t_0,T]} \norm{ \alpha_t}_{\mathcal{C}^3}$ and $T >0$ such that 
\begin{equation} 
\sup_{t \in [t_0,T]} \sup_{ \mu \in \mathcal{P}(\T^d)} \sup_{(x,y) \in (\T^d)^2} \bigl| \nabla^2_{x,y} \frac{\delta^2 \Phi }{ \delta \mu^2} (t,\mu,x,y) \bigr| \leq C.
\label{eq:boundlineareqn28/04}
\end{equation}
Compared to the setting of \cite{BuckdahnLiPengRainer} and \cite{CST} our state equation involves a zero-th order term but the same method applies. Since, for all $(t,\mu) \in [t_0,T] \times \mathcal{P}(\T^d)$,
$$ -\int_{\T^d} \alpha_t(x) \cdot \nabla_x \frac{\delta \Phi}{\delta \mu}(t,\mu,x) - \frac{1}{2} \int_{\T^d} |\alpha_t(x)|^2 d\mu_t(x) \leq \frac{1}{2} \int_{\T^d} \bigl| \nabla_x \frac{\delta\Phi}{\delta \mu}(t,\mu,x) \bigr|^2 d\mu(x), $$
$\Phi$ is a classical super-solution, over $[t_0,T] \times \mathcal{P}(\T^d)$, of the Hamilton-Jacobi equation \eqref{eq:MasterHJB16/04}. We can apply the projection argument of Section \ref{sec:formalargumentprojection} to $\Phi^N : [t_0,T] \times (\T^d)^N \times \R^N \rightarrow \R$ defined for all $(t,\bx,\ba) \in [t_0,T] \times (\T^d)^N \times \R^N$ by
$$ \Phi^N(t, \bx,\ba) = \Phi(t, \mu^N_{\bx,\ba}),$$
and we find that $\Phi^N$ is a classical super-solution to
\begin{equation}
\left \{
\begin{array}{ll}
 \displaystyle    -\partial_t \Phi^N - \frac{1}{2}\sum_{i=1}^N \Delta_{x^{i}} \Phi^N - \sum_{i=1}^N V(x^{i}) \partial_{a^{i}} \Phi^N + \sum_{i=1}^N \frac{n^{i}[\ba]}{2} |\nabla_{x^{i}} \Phi^N|^2 \\
\displaystyle \hspace{120pt}\geq -  \sum_{i=1}^N \bigl( \frac{1}{n^{i}[\ba]} \bigr)^2 \tr \nabla^2_{xy} \frac{\delta^2 \Phi}{\delta \mu^2}(t, \mu^N_{\bx,\ba},x^{i},x^{i})   , \quad \mbox{ in } [t_0,T]\times (\T^d)^N \times (\R)^N \\ 
\displaystyle    \Phi^N(T, \bx, \ba) = \sum_{i=1}^N \frac{1}{n^{i}[\ba]} g(x^{i}), \quad \mbox{ in } (\T^d)^N \times (\R)^N. 
\end{array}
\right.
\end{equation}
Arguing similarly as in the proof of Proposition \ref{prop:hardineq19/04}, see also Section \ref{sec:formalargumentprojection}, we conclude that there is a constant $C>0$ depending on $T$, $\norm{V}_{\mathcal{C}^0}$ and the constant appearing in  \eqref{eq:boundlineareqn28/04} such that, for all $(t,\bx,\ba) \in [t_0,T] \times (\T^d)^N \times \R^N$, 
$$ v^n(t,\bx,\ba) \leq \Phi(t, \mu^N_{\bx,\ba}) +  C \sum_{i=1}^N \bigl( \frac{1}{n^{i}[\ba]} \bigr)^2. $$
In particular, recalling that $\Phi(t_0, \mu^N_{\bx_0,\ba_0}) =  \mathcal{V}(t_0, \mu^N_{\bx_0,\ba_0}) $ because $\alpha$ is an optimal control for $\mathcal{V}(t_0,\mu^N_{\bx_0,\ba_0}) $ we get 
$$ v^n(t_0,\bx_0,\ba_0) \leq \mathcal{V}(t_0, \mu^N_{\bx_0,\ba_0}) +  C \sum_{i=1}^N \bigl( \frac{1}{n^{i}[\ba_0]} \bigr)^2. $$
This concludes the proof of the Proposition.
\end{proof}

\section{Existence of optimal controls and optimality conditions}

\label{sec:FOC}

The goal of this section is to prove Theorem \ref{thm:OptimalityConditons}.  We need to extend some results of \cite{CLL23}, especially concerning the existence of smooth feedback optimal controls. Our main innovation compared to \cite{CLL23} is to prove that bounded optimal controls for \eqref{def:limitvfV} exist, and any of them satisfies the optimality conditions. This is an improvement over \cite{CLL23} where the optimality conditions are obtained under the condition that the optimal control is bounded. This improvement is necessary to understand the regularity of the value function in Section \ref{sec:AnalysisVF}.

We need a preliminary well-posedness result for the backward Hamilton-Jacobi equation. For all $R>0$ and $p \in \R^d$ we denote 
\begin{equation}
H^R(p) := \sup_{|a| \leq R}  \bigl \{ -a\cdot p - \frac{1}{2}|a|^2 \bigr \} = \left \{ \begin{array}{ll} \frac{1}{2} |p|^2 & \mbox{if } |p| \leq R \\ R|p| - \frac{1}{2}|R|^2, & \mbox{if } |p| \geq R, \end{array} \right.
\label{eq:defnHR}
\end{equation}
and we notice that $H^R$ is continuously differentiable and $\nabla_pH^R$ is globally Lipschitz and bounded. For convenience we define as well $H^{\infty}(p) = \frac{1}{2}|p|^2$.

\begin{lem}
\label{lem:AprioriEstimatesHJB14/04}
    Take $\mu \in \mathcal{C}([t_0,T], \mathcal{P}(\T^d))$ and assume that $V$ and $g$ are in $\mathcal{C}^2(\T^d)$. For all $R\in \R^+ \cup \{+\infty \}$, there is a unique solution $u^R$ in $\mathcal{C}^{1,2}([t_0,T] \times \T^d)$ to the HJB equation
\begin{equation}
    -\partial_t u_t^R + H^R( \nabla u_t^R) - \frac{1}{2} \Delta u_t^R + \bigl( V-\langle V; \mu_t \rangle ) u_t^R - V \langle u_t^R;\mu_t \rangle = 0, \quad \mbox{ in } [t_0,T] \times \T^d, \quad \quad u_T^R = g \mbox{ in } \T^d.
\label{eq:HJBR}
\end{equation}
Moreover, there is a non-decreasing function $\Lambda : \R^+ \rightarrow \R^+$, independent of $(t_0,T,\mu,g,V,R)$ such that 
$$ \sup_{t \in [t_0,T]} \norm{ u^R_t}_{\mathcal{C}^1(\T^d)} \leq \Lambda \bigl( T + \norm{V}_{\mathcal{C}^1(\T^d)} + \norm{g}_{\mathcal{C}^1(\T^d)} \bigr).$$
In particular, there is $R_0>0$ depending only on $(T, \norm{V}_{\mathcal{C}^1(\T^d)}, \norm{ g}_{\mathcal{C}^1(\T^d)})$ such that, for all $R >0$, $u^R$ is $R_0$-Lipschitz continuous in space, uniformly in time and therefore, as soon as $R \geq R_0$, $u^R$ coincides with $u$, the unique classical solution to 
\begin{equation}
    -\partial_t u_t + \frac{1}{2}|\nabla u_t|^2 - \frac{1}{2} \Delta u_t + \bigl( V-\langle V; \mu_t \rangle ) u_t - V \langle u_t;\mu_t \rangle = 0, \quad \mbox{ in } [t_0,T] \times \T^d, \quad \quad u_T = g \mbox{ in } \T^d.
\label{eq:backwardHJB}
\end{equation}
\end{lem}

\begin{proof}
We only give a sketch of the proof. Existence and uniqueness for \eqref{eq:HJBR} follow from the regularity of $H^R,V$ and $g$, the non-degeneracy of the second order term, as well as a fixed point argument to take care of the non local term, see \cite[Section 3.6.1 ]{CLL23} for the fixed point argument. We focus on the $L^{\infty}$ and Lipschitz estimates to prove that they are uniform in $R$ for $R$ large enough. We start with the former. For $R>0$ and $\delta, \gamma >0$ we let, for all $(t,x) \in (0,T] \times \T^d$,  
$$v_t(x):=e^{-\int_t^T \langle V; \mu_s \rangle ds}u_t^R(x) - \delta (T-t) - \frac{\gamma}{t}
$$ 
and we observe that $v$ solves 
$$
- \partial_t v_t + e^{- \int_t^T \langle V; \mu_s \rangle ds } H^R( e^{ \int_t^T \langle V;\mu_s \rangle ds} \nabla v_t) - \frac{1}{2}\Delta v_t + V v_t - V \langle v_t;\mu_t \rangle = - \delta - \frac{\gamma}{t^2} \quad \mbox{ in } (t_0,T] \times \T^d, .
$$
with $ v_T = g - \gamma/T$. Let $( \bar{t},\bar{x}) \in (0,T] \times \T^d$ be a maximum point of $v$. If $ \bar t \neq T$, we must have $ \partial_t v_{\bar{t}} (\bar{x}) = 0$, $\nabla v_{\bar{t}}(\bar{x}) =0$ and $\Delta v_{\bar{t}} (\bar{x}) \leq 0$. We deduce from the equation that
$$ V(\bar{x}) \bigl(  v_{\bar{t}}( \bar{x})  - \langle v_{\bar{t}}; \mu_t \rangle) \leq -\delta - \gamma/t^2 <0, $$
But $V \geq 0$ and $v_{\bar{t}}(\bar{x}) \geq \langle v_{\bar{t}} ; \mu_{\bar{t}} \rangle $ since $\bar{x}$ is a maximum point for $v_{\bar{t}}$ and $\bar{\mu}_t$ is a probability measure. As a consequence, $V(\bar{x})( v_{\bar{t}}(\bar{x}) - \langle v_{\bar{t}}; \mu_{\bar{t}} \rangle ) \geq 0$ and we get a contradiction. So $\bar{t}$ must be equal to $T$ which means that, for all $(t,x) \in [0,T] \times \T^d$, $v_t(x) \leq v_T( \bar{x}) \leq \norm{g}_{\mathcal{C}^0} - \gamma/T.$ If we let $\delta, \gamma \rightarrow 0$ and recall the definition of $v$ we obtain
$$ 
u^R(t,x) \leq e^{\int_t^T \langle V; \mu_s \rangle ds} \norm{ g }_{\mathcal{C}^0} \quad  (t,x) \in [0,T] \times \T^d.
$$
Following the same reasoning, looking for minimum points of $v^{-} := e^{-\int_t^T \langle V; \mu_s \rangle}u^R + \delta (T-t) + \gamma/t$ we obtain the symmetric lower bound on $u^R$ and finally we deduce that 
\begin{equation} 
-e^{\int_t^T \langle V ; \mu_s \rangle ds } \norm{ g}_{\mathcal{C}^0} \leq u^R_t(x) \leq e^{\int_t^T \langle V ; \mu_s \rangle ds } \norm{ g}_{\mathcal{C}^0} \quad \quad (t,x) \in [0,T] \times \T^d.  
\label{eq:linftyuR}
\end{equation}
Notice in particular that the bound does not depend on $R >0$. For the Lipschitz estimate we proceed as follows. For some $\gamma \in \R$ to be determined later, we let $ w(t,x) := e^{\gamma(t-T)} \frac{1}{2} |\nabla u^R(t,x)|^2 $. Differentiating in space the HJB equation and taking inner product with $e^{\gamma(t-T)} \nabla u^R(t,x)$ leads to
\begin{align*}
-\partial_t w_t + &H_p^R( \nabla u_t^R) \cdot \nabla w_t - \frac{1}{2} \Delta w_t \\
&= -e^{\gamma (t-T)} |\nabla^2 w_t|^2 - (\gamma + 2(V-\langle V ; \mu_t \rangle) ) w_t + ( \langle u_t ;\mu_t \rangle -u_t) \nabla V \cdot \nabla u_t e^{\gamma(t-T)}.
\end{align*}
As a consequence, as soon as $\gamma \geq 4 \norm{V}_{\mathcal{C}^0(\T^d)}$, $ -(\gamma + 2(V(x)- \langle V; \mu_t \rangle )) w_t(x) \leq 0$ for all $(t,x) \in [t_0,T] \times \T^d$ and  we obtain, for some constant $C>0$ depending on $T$, $\norm{V}_{\mathcal{C}^1(\T^d)}$, and $\norm{u^R}_{\mathcal{C}^0(\T^d)}$ (and therefore independent from $R$, thanks to \eqref{eq:linftyuR}),
$$-\partial_t w_t + H_p^R(\nabla u_t^R) \cdot \nabla w_t - \frac{1}{2} \Delta w_t \leq C \sqrt{ \norm{w}_{\mathcal{C}^{0}}}. $$
By comparison with the obvious super solution
$$ (t,x) \mapsto \frac{1}{2} \norm{ \nabla g}^2_{\mathcal{C}^{0}} + C(T-t) \sqrt{ \norm{w}_{\mathcal{C}^{0}} },$$
we obtain, for all $(t,x) \in [0,T] \times \T^d,$
 $$ w(t,x) \leq \frac{1}{2} \norm{\nabla g}^2_{\mathcal{C}^{0}} + CT \sqrt{ \norm{w}_{\mathcal{C}^{0}} }.$$
 Taking the supremum in $(t,x)$ yields (recall that $w \geq 0$)
 $$ \norm{ w}_{\mathcal{C}^{0}} \leq \frac{1}{2} \norm{\nabla g}^2_{\mathcal{C}^{0}} + CT \sqrt{ \norm{w}_{\mathcal{C}^{0}} },$$
 which implies
\begin{equation}
    \label{eq:LipesstimateuR}
\norm{ w}_{\mathcal{C}^{0}} \leq  C^2T^2 + \norm{\nabla g}^2_{\mathcal{C}^{0}}
\end{equation}
and from which we deduce that $\norm{w}_{\mathcal{C}^{0}}$, and therefore $\norm{ \nabla u^R}_{\mathcal{C}^{0}}$, are bounded by a constant depending on $\norm{ \nabla g}_{\mathcal{C}^{0}}$,  $\norm{V}_{\mathcal{C}^{0}}$, $\norm{\nabla V}_{\mathcal{C}^{0}}$ and $\norm{u^R}_{\mathcal{C}^{0}}$. In particular, $u^R$ and $\nabla u^R$ are bounded independently of  $R$. Take $R_0$ such that, for all $R >0$, $\norm{\nabla u^R}_{\mathcal{C}^{0}} \leq R_0$. Then, recalling the definition of $H^R$ in \eqref{eq:defnHR}, we get 
 $$  H^R (\nabla u^R) = \frac{1}{2} |\nabla u^R|^2, \quad \forall R \geq R_0.$$
 This means that, for all $R \geq R_0$, $u^R$ solves \eqref{eq:backwardHJB} and therefore $u^R = u$ the unique solution to \eqref{eq:backwardHJB}.
\end{proof}

We go on with the optimality conditions. First we fix $R>0$ and we investigate the control problem 
\begin{equation} 
\inf_{(\mu,\alpha)} \int_{t_0}^T \int_{\T^d} \frac{1}{2} |\alpha_t(x)|^2 d\mu_t(x)dt + \int_{\T^d} g(x)d \mu_T(x),
\label{PbR}
\tag{Pb(R)}
\end{equation}
where $\mu$ belongs to $\mathcal{C}([t_0,T], \mathcal{P}(\T^d))$ and $ \alpha $ to $ L^{\infty}([t_0,T] \times \T^d, \mu_t \otimes dt; \R^d)$, satisfying together the Fokker-Planck equation  
\begin{equation} 
\partial_t \mu - \frac{1}{2} \Delta\mu +\div(\alpha \mu) + (V - \langle V, \mu \rangle)\mu = 0, \mbox{ in } (t_0,T) \times \T^d, \quad \mu(t_0) = \mu_0,
\label{eq:conditionalFPE}
\end{equation}
and the constraint 
\begin{equation}
\mu_t \otimes dt  \{ (t,x) \in [t_0,T] \times \T^d , \quad |\alpha_t(x) | > R  \} = 0. 
\label{eq:Linftybound14/04}
\end{equation}

We recall from \cite{CLL23}, that $\mu$ has the following probabilistic representation
\begin{equation}
\mu_t(dx) = \E \bigl[ \delta_{X_t} q_t  \bigr] \quad \mbox{ with } \quad q_t := \frac{\exp \bigl( - \int_{t_0}^t V(X_s)ds \bigr)}{ \E \exp \bigl( - \int_{t_0}^t V(X_s)ds \bigr) } \quad \mbox{ for  } t \in [t_0,T] 
\label{eq:stochasticrepresentation}
\end{equation}
where $(X_t)_{t \in [t_0,T]}$ is the solution to 
$$ dX_t = \alpha_t(X_t) dt + dB_t.$$

From this representation we obtain the following time regularity for $\mu$. We denote by $d_1$ the first Wasserstein distance over $\mathcal{P}(\T^d)$. 

\begin{lem}
    Let $(\mu,\alpha)$ be a solution to \eqref{eq:conditionalFPE}. There is a non decreasing function $\Lambda : \R^+ \rightarrow \R^+$ independent from $\alpha,V,T$ such that
    $$ \sup_{t_1 < t_2 \in [t_0,T]} \frac{d_1 (\mu_{t_2}, \mu_{t_1})}{\sqrt{t_2- t_1}} \leq \Lambda \Bigl( \int_{t_0}^T \int_{\T^d} |\alpha_t(x)|^2 d\mu_t(x)dt + \norm{V}_{\mathcal{C}^0} + T \Bigr). $$
\end{lem}

\begin{proof}
    Take $t_1,t_2 \in [t_0,T]$ with $t_1 < t_2$. Let $\phi : \T^d \rightarrow \R$ be a $1$-Lipschitz function with $\phi(0)=0$. In particular, $|\phi(x)| \leq \sqrt{d}$ over $\T^d$. Then, 
\begin{align*}
    \int_{\T^d} \phi(x) d(\mu_{t_2} - \mu_{t_1})(x) &= \E \bigl[ \phi(X_{t_2}) q_{t_2} - \phi(X_{t_1}) q_{t_1}] \\
    &= \E \bigl[ \bigl(\phi(X_{t_2}) - \phi(X_{t_1}) \bigr) q_{t_2}] + \E \bigl[  \phi(X_{t_1}) (q_{t_2} - q_{t_1}) ] \\
    &\leq \E \bigl[ |X_{t_2}- X_{t_1}|^2 \bigr]^{1/2} \E \bigl[ q_{t_2}^2 \bigr]^{1/2} + \sqrt{d}\; \E \bigl[ |q_{t_2} - q_{t_1} | \bigr].
\end{align*}
On the one hand, we easily find from the definition of $q_t$ in \eqref{eq:stochasticrepresentation}, that 
$$ \sup_{t_2 \in [t_0,T]} \E \bigl[ q_{t_2}^2 \bigr]^{1/2} + \sup_{t_1 , t_2 \in [t_0,T]} \frac{\E \bigl[|q_{t_2} - q_{t_1}|]}{|t_2 -t_1|} \leq \Lambda \bigl( T + \norm{V}_{\mathcal{C}^0} \bigr), $$
for a non-decreasing function $\Lambda$ independent from $T$ and $V$. On the other hand, using the stochastic differential equation we get 
\begin{align*}
    \E \bigl[ |X_{t_2} - X_{t_1} |^2 \bigr]^{1/2} & \leq \E \bigl[ | \int_{t_1}^{t_2} \alpha_t(X_t) dt |^2 \bigr]^{1/2} + \E \bigl[ |B_{t_2} - B_{t_1} |^2 \bigr]^{1/2}.
\end{align*}
Cauchy-Schwarz inequality then gives
$$ \bigl| \int_{t_1}^{t_2} \alpha_t(X_t)dt \bigr|^2  \leq (t_2 - t_1) \int_{t_1}^{t_2} |\alpha_t(X_t)|^2 dt $$
and therefore
$$ \E \bigl[ |X_{t_2} - X_{t_1} |^2 \bigr]^{1/2} \leq \sqrt{t_2 -t_1} \E \bigl[ \int_{t_0}^T |\alpha_t(X_t)|^2 dt \bigr] + \sqrt{t_2- t_1}.$$
It remains to notice that 
$$ \E\bigl[\int_{t_0}^T |\alpha_t(X_t)|^2 dt \bigr] \leq \sup_{t \in [t_0,T]}\norm{ q_t^{-1}}_{L^{\infty}} \E \bigl[ \int_{t_0}^T |\alpha_t(X_t)|^2 q_tdt]  = \sup_{t \in [t_0,T]}\norm{ q_t^{-1}}_{L^{\infty}} \int_{t_0}^T \int_{\T^d} |\alpha_t(x)|^2 d\mu_t(x)dt  $$
and $$ \sup_{t \in [t_0,T]}\norm{ q_t^{-1}}_{L^{\infty}} \leq \Lambda( \norm{V}_{\mathcal{C}^0} + T \bigr) $$
for possibly another non-decreasing function $\Lambda$. All in all we find that
$$ \int_{\T^d} \phi(x) d(\mu_{t_2} - \mu_{t_1})(x) \leq \Lambda\Bigl( \int_{t_0}^T \int_{\T^d} |\alpha_t(x)|^2 d\mu_t(x)dt + \norm{V}_{\mathcal{C}^0} + T \Bigr) \sqrt{t_2-t_1} $$
for a new non-decreasing function $\Lambda$ independent from $T,V,(\mu,\alpha)$. The result follows by taking the supremum over the set of $1$-Lipschitz functions $\phi$.
\end{proof}

\begin{lem}
    Optimal solutions for Problem \ref{PbR}  exist and one of them 
satisfies the optimality conditions $ \alpha_t = - H_p^R(\nabla u_t) $ for some solution $(\mu_t, u_t)$ of the PDE system
    \begin{equation}
\left \{
    \begin{array}{ll}
    -\partial_t u_t - \frac{1}{2} \Delta u_t + H^R(\nabla u_t) + (V- \langle V; \mu_t \rangle) u_t  - V \langle u_t, \mu_t \rangle = 0 \quad \mbox{ in } (t_0,T) \times \T^d, \\
    \partial \mu_t - \frac{1}{2} \Delta u_t - \div(H^R_p(\nabla u_t) \mu_t) + (V- \langle V,\mu_t \rangle)\mu_t = 0 \quad \mbox{ in } (t_0,T) \times \T^d, \\
    \mu_{t_0} = \mu_0, \quad u_T = g \quad \quad \mbox{ in } \T^d.
    \end{array}
\right.
\end{equation}
\end{lem}

\begin{rmk}
    By Lemme \ref{lem:AprioriEstimatesHJB14/04}, the optimal control $\alpha$ is bounded independently of $R >0$. In particular, $\alpha_t = -\nabla u_t $ and $H^R (\nabla u_t) = \frac{1}{2} |\nabla u_t|^2$ as soon as $R >R_0$ with $R_0$ beind defined in Lemma \ref{lem:AprioriEstimatesHJB14/04}.
\end{rmk}

\begin{rmk}
    For now it is enough to prove that \textit{at least} one solution satisfies the optimality conditions. The general statement will be proven next.
\end{rmk}

\begin{proof}
\textit{Step 1. Existence.}  

We take a minmizing sequence $(\mu^n, \alpha^n)_{n \in \mathbb{N}}$. In particular, thanks to the previous lemma and the fact that $g$ is bounded over $\T^d$, we have the uniform bound
$$ \sup_{ n \in \mathbb{N}} \Bigl \{  \int_{t_0}^T \int_{\T^d} |\alpha^n_t(x)|^2 d\mu^n_t(x)dt + \sup_{t_1 <t_2 \in [t_0,T]} \frac{d_1(\mu^n_{t_2}, \mu^n_{t_1})}{\sqrt{t_2-t_1}}  \Bigr\} < +\infty. $$
We can then proceed as in \cite[Proposition 1.2]{Daudin2021}, to find $\mu \in \mathcal{C}^{1/2} \bigl([t_0,T], \bigl(\mathcal{P}(\T^d),d_1 \bigr) \bigr)$ and $\alpha \in L^2([t_0,T] \times \T^d, \mu_t \otimes dt; \R^d)$ such that $(\mu,\alpha)$ solves \eqref{eq:backwardHJB}, $\mu^n$ converges to $\mu$ in $\mathcal{C}([t_0,T], \mathcal{P}(\T^d))$, the finite vector-valued measures $(\alpha^n \mu_t^n \otimes dt)$ weak-$*$ converge to $\alpha \mu_t \otimes dt \in \mathcal{M}([t_0,T] \times \T^d; \R^d)$  and 
$$ \int_{t_0}^T \int_{\T^d} |\alpha_t(x)|^2 d\mu_t(x)dt \leq \liminf_{n \rightarrow +\infty} \int_{t_0}^T \int_{\T^d} |\alpha^n_t(x)|^2 d\mu^n_t(x)dt. $$
In particular, since $g$ is continuous, the terminal costs converge and we have
$$ J \bigl( (\mu , \alpha) \bigr) \leq \liminf_{n \rightarrow +\infty} J \bigl( (\mu^n , \alpha^n) \bigr).$$
The only point we need to check is that we can find such a limit point $(\mu,\alpha)$ satisfying the bound \eqref{eq:Linftybound14/04}. To this end we notice that, for all continuous function $f : [t_0,T] \times \T^d \rightarrow \R^d$ we have
$$ \int_{t_0}^T \int_{\T^d} \alpha^n_t(x) \cdot f_t(x) d\mu^n_t(x) \leq R \int_{t_0}^T \int_{\T^d} |f_t(x)| d\mu_t^n(x)dt.  $$
We can let $n \rightarrow +\infty$ in the inequality above to deduce that, for all continuous function $f : [t_0,T] \times \T^d \rightarrow \R^d$,
$$ \int_{t_0}^T \int_{\T^d} \alpha_t(x) \cdot f_t(x) d\mu_t(x) dt \leq R \int_{t_0}^T \int_{\T^d} |f_t(x)|d\mu_t(x)dt$$
and we deduce that $| \alpha(t,x)| \leq R$, $\mu_t \otimes dt$-almost-everywhere. 

\textit{Step 2. Variations.} This step is very similar to \cite[Section 3.6.4]{CLL23}. Thanks to \textit{Step 1} we can consider  a solution $(\mu, \alpha)$ (which a priori depends on $R$). Up to replacing $\alpha$ by $(t,x) \mapsto \alpha_t(x) \mathbf{1}_{|\alpha_t(x)| \leq R} $ we can assume that $|\alpha_t(x)| \leq R$ everywhere.

Now we take a measurable map $\beta : [t_0,T] \times \T^d \rightarrow \R^d$ with $ |\beta_t(x)| \leq R $ for all $(t,x) \in [t_0,T] \times \R^d$. For very $\epsilon \in [0,1]$ we let $\alpha^{\epsilon} := (1-\epsilon) \alpha + \epsilon \beta$ and $\mu^{\epsilon}$ the solution to
$$ \partial_t \mu_t^{\epsilon} + \div(\alpha_t^{\epsilon}\mu_t^{\epsilon}) - \frac{1}{2} \Delta \mu_t^{\epsilon} + (V-\langle V;\mu_t^{\epsilon} \rangle) \mu_t^{\epsilon} =0, \quad \mu^{\epsilon}_{t_0} = \mu_0,$$
given by the stochastic representation  \eqref{eq:stochasticrepresentation}.  
Following precisely \cite[Section 3.6.3]{CLL23}  we know that 
$$ \frac{d}{d\epsilon}\Big|_{\epsilon =0^+} J(\mu^{\epsilon}, \alpha^{\epsilon}) = \int_{t_0}^T \int_{\T^d} (\beta_t -\alpha_t)(x) \cdot ( \alpha_t(x) + \nabla u_t(x)) d\mu_t(x)dt.$$
where $u$ is solution to the adjoint equation (whose well-posedness is also discussed in \cite{CLL23})
$$-\partial_t u_t - \frac{1}{2} \Delta u_t- \alpha_t \cdot \nabla u_t + (V-\langle V; \mu_t \rangle)u_t + V \langle u_t,\mu_t \rangle = \frac{1}{2} |\alpha_t|^2, \quad u_T = g. $$
By minimality of $\alpha$ we deduce that, for all such $\beta$, it holds 
$$ \int_{t_0}^T \int_{\T^d} (\beta_t(x) - \alpha_t(x)) \cdot ( \alpha_t(x) + \nabla u_t(x) ) d\mu_t(x) dt \geq 0.$$
Rearranging and using Young's inequality $ \beta_t \cdot \alpha_t \leq \frac{1}{2} |\alpha_t|^2 + \frac{1}{2} |\beta_t|^2$ leads to
$$ \int_{t_0}^T \int_{\T^d} \bigl( - \beta_t(x) \cdot \nabla u_t(x) - \frac{1}{2} |\beta_t(x)|^2 \bigl) d\mu_t(x)dt \leq  \int_{t_0}^T \int_{\T^d} \bigl( - \alpha_t(x) \cdot \nabla u_t(x) - \frac{1}{2} |\alpha_t(x)|^2 \bigl) d\mu_t(x)dt. $$
Taking the supremum over $\beta$ and recalling the definition of $H^R$ we obtain 
$$\int_{t_0}^T \int_{\T^d} H^R( \nabla u_t(x)) d\mu_t(x) dt \leq  \int_{t_0}^T \int_{\T^d} \bigl( - \alpha_t(x) \cdot \nabla u_t(x) - \frac{1}{2} |\alpha_t(x)|^2 \bigl) d\mu_t(x)dt. $$
However, for all $(t,x) \in [t_0,T] \times \R^d$ we have 
$$- \alpha_t(x) \cdot \nabla u_t(x) - \frac{1}{2} |\alpha_t(x)|^2 \leq H^R( \nabla u_t(x)),  $$
and therefore, we can conclude that
$$ - \alpha_t(x)\cdot \nabla u_t(x) - \frac{1}{2} |\alpha_t(x)|^2 = H^R( \nabla u_t(x)), \quad \mu_t \otimes dt - \mbox{ almost-everywhere.}$$
Equivalently, in the support of $\mu_t \otimes dt$ we have
$$ \alpha_t(x) = -H^R_p( \nabla u_t(x)).  $$
Collecting the equations for $\mu$ and $u$ we have found an optimal control $\alpha$ with associated curve $\mu$ and multiplier $u$ such that $\alpha = -H_p^R( \nabla u)$ and $(\mu, u)$ solves the coupled pde system 
\begin{equation}
\left \{
    \begin{array}{ll}
    -\partial_t u_t - \frac{1}{2} \Delta u_t + H^R( \nabla u_t) + (V- \langle V; \mu_t \rangle) u_t  - V \langle u_t; \mu_t \rangle = 0 \quad \mbox{ in } (t_0,T) \times \T^d, \\
    \partial \mu_t - \frac{1}{2} \Delta u_t - \div(H^R_p(\nabla u_t) \mu_t) + (V- \langle V;\mu_t \rangle)\mu_t = 0 \quad \mbox{ in } (t_0,T) \times \T^d, \\
    \mu_{t_0} = \mu_0, \quad u_T = g \quad \quad\mbox{ in } \T^d.
    \end{array}
\right.
\end{equation}
In particular, $\mu$ is driven by a Lipschitz continuous vector field and therefore, for all $t \in (t_0,T]$, $\mu_t \otimes dt$ has full support and the equations are satisfied everywhere.  This concludes the proof of the Lemma.
\end{proof}

We can finally conclude with the Proof of Theorem \ref{thm:OptimalityConditons}.

\begin{proof}[Proof of Theorem \eqref{thm:OptimalityConditons}] Now we can apply Lemma  \ref{lem:AprioriEstimatesHJB14/04} to find $R_0>0$ such that, for any $R >0$ there is a solution $\alpha^R$ to the control problem \eqref{PbR} such that $| \alpha_t(x)^R| \leq R_0$ for all $(t,x) \in [t_0,T] \times \T^d$. Let $\mathcal{V}^R(t_0,\mu_0)$ be the minimal value of \eqref{PbR}. We have proved that, for all $R \geq R_0,$ $\mathcal{V}^R = \mathcal{V}^{R_0}$ (we don't do any better by allowing controls with $L^{\infty}$-norm greater than $R$). By definition of $\mathcal{V}$ and $\mathcal{V}^R$, we have $ \lim_{R \rightarrow +\infty} \mathcal{V}^R(t_0,\mu_0) = \mathcal{V}(t_0,\mu_0)$.
Notice that this limit was not trivial if the original problem was defined over controls in $L^2([t_0,T] \times \T^d, \mu_t \otimes dt; \R^d)$. However, with controls in $L^{\infty}$ it is plain. Hence, we deduce that $\mathcal{V}(t_0,\mu_0) = \mathcal{V}^{R_0}(t_0,\mu_0)$ and therefore there is at least one solution to the original problem satisfying the optimality conditions.

It remains to justify that \textit{any} solution to the original problem must satisfy the optimality conditions.  
The argument is similar to \cite[Lemma 3.1]{Daudin2021}. We take $(\tilde{\mu}, \tilde{\alpha})$ a solution to the original problem \eqref{def:limitvfV} and we consider a new problem 
\begin{equation}
\inf_{(\mu, \alpha)} J(\mu, \alpha) + \int_{t_0}^T q(\mu_t, \tilde{\mu_t})dt
\label{newproblem}
\end{equation}
for some smooth (square)-distance-like function $q: \mathcal{P}(\T^d) \times \mathcal{P}(\T^d) \rightarrow \R$ satisfying $q (\mu^1, \mu^2) \geq 0$ and $q(\mu^1, \mu^2) = 0,$ if and only if $\mu^1 = \mu^2$ as well as 
\begin{equation} 
\frac{\delta q}{\delta \mu^1}(\mu^1, \mu^1,x)= 0 \quad \forall \mu^1 \in \mathcal{P}(\T^d) \quad \forall x \in \T^d.
\label{eq:proplinderivative14/04}
\end{equation}
For instance $q$ can be the squared negative Sobolev norm 
$$ q(\mu_1, \mu_2) := \norm{\mu_2 - \mu_1}^2_{H^{-k}} $$
with $k > d/2 +2$.

Since $(\tilde{\mu}, \tilde{\alpha})$ is a minimum of $J$ and $\int_{t_0}^T q(\mu_t, \tilde{\mu}_t)dt >0$ whenever $\mu \neq \tilde{\mu},$ any minimizer $(\mu, \alpha)$ of the new cost must verify $\mu = \tilde{\mu}.$ We take $(\tilde{\mu}, \alpha)$ another minimizer, $\lambda \in [0,1]$ and let $\alpha^{\lambda} := (1- \lambda) \tilde{\alpha} + \lambda \alpha$. We easily check that for all $\lambda \in [0,1]$ $(\tilde{\mu}, \alpha^{\lambda})$ solves the Fokker-Planck equation and is admissible. Comparing the costs we find that
$$ \int_{t_0}^T \int_{\T^d} |\alpha_t^{\lambda}(x)|^2 d\tilde{\mu}_t(x) dt =  \int_{t_0}^T \int_{\T^d} |\tilde{\alpha}_t(x)|^2 d\tilde{\mu}_t(x) dt, \quad \forall \lambda \in [0,1]. $$
By strict convexity of the square function we deduce that $ \alpha = \tilde{\alpha}, $ $ \tilde{\mu}_t \otimes dt$-  almost-everywhere. As a conclusion Problem \eqref{newproblem} admits exactly one solution which is $(\tilde{\mu},\tilde{\alpha})$. We argue as before for Problem \eqref{newproblem} and deduce that $|\tilde{\alpha}| \leq R_0$ and $(\tilde{\mu}, \tilde{\alpha})$ satisfies the optimality conditions. Notice that the distance-like function $q$ does not appear in the optimality conditions thanks to the property \eqref{eq:proplinderivative14/04} of its linear derivative.

To conclude the proof of the theorem, it remains to justify that $u_t$ satisfies the estimate \eqref{eq:estimateforu15/04}. This is easily done by induction, using Duhamel's representation formula, see  \cite[Lemma A.3]{ddj2023}.
\end{proof}

\bibliographystyle{plain}

\end{document}